\documentclass{IJCAS}

 \usepackage{graphicx}  
\usepackage{amsmath}
\usepackage{url}
\usepackage{array,tabularx}
\usepackage{cite}  
\usepackage{subfigure,color,url}
\usepackage{dsfont}

\journaltitle{Submission to International Journal of Control, Automation, and Systems}
\setarticlestartpagenumber{1}
\newcommand{\sat}{\mathop{\mathrm{sat}}}
\newcommand{\trs}{\mathop{\mathrm{tr}}}
\newcommand{\satr}{\mathop{\mathrm{sat}}}
\newcommand{\norm}[1]{\ensuremath{\left\| #1 \right\|}}

\newcommand{\bracket}[1]{\ensuremath{\left[ #1 \right]}}
\newcommand{\braces}[1]{\ensuremath{\left\{ #1 \right\}}}

\newcommand{\refeqn}[1]{(\ref{eqn:#1})}

\newcommand{\SO}{\ensuremath{\mathsf{SO(3)}}}

\newcommand{\so}{\ensuremath{\mathfrak{so}(3)}}

\renewcommand{\Re}{\ensuremath{\mathbb{R}}}
\newcommand{\Sph}{\ensuremath{\mathsf{S}}}

\newcommand{\D}{\ensuremath{\mathbf{D}}}

\newcommand{\g}{\ensuremath{\mathfrak{g}}}

\newtheorem{prop}{Proposition}

\begin{document}


\title{Geometric Control of a Quadrotor UAV Transporting a Payload Connected via Flexible Cable}

\author{Farhad A. Goodarzi, Daewon Lee, and Taeyoung Lee
		\thanks{Farhad A. Goodarzi, Daewon Lee, and Taeyoung Lee are with the School of Mechanical and aerospace
	Engineering, The George Washington University, Washington DC 20052
	 (e-mail: \{fgoodarzi,daewonlee,tylee\}@gwu.edu).}
	\thanks{\textsuperscript{\footnotesize\ensuremath{*}}This research has been supported in part by NSF under the grants CMMI-1243000 (transferred from 1029551), CMMI-1335008, and CNS-1337722.
}
}

\begin{abstract}
We derived a coordinate-free form of equations of motion for a complete model of a quadrotor UAV with a payload which is connected via a flexible cable according to Lagrangian mechanics on a manifold. The flexible cable is modeled as a system of serially-connected links and has been considered in the full dynamic model. A geometric nonlinear control system is presented to exponentially stabilize the position of the quadrotor while aligning the links to the vertical direction below the quadrotor. Numerical simulation and experimental results are presented and a rigorous stability analysis is provided to confirm the accuracy of our derivations. These results will be particularly useful for aggressive load transportation that involves large deformation of the cable.
\end{abstract}


\maketitle

\newtheorem{theorem}{Theorem}
\newtheorem{lemma}{Lemma}


\section{Introduction}

Unmanned aerial vehicles (UAV) have been studied for different applications such as surveillance or mobile sensor networks as well as for educational purposes. Quadrotors are one kind of these UAVs which are very popular due to their dynamic simplicity, maneuverability and high performance. 
Areal transportation of a cable-suspended load has been studied traditionally for helicopters~\cite{CicKanJAHS95,BerPICRA09}. 
Recently, small-size single or multiple autonomous vehicles are considered for load transportation and deployment~\cite{PalCruIRAM12,MicFinAR11,MazKonJIRS10,MelShoDARSSTAR13}, and trajectories with minimum swing and oscillation of payload are generated~\cite{ ZamStaJDSMC08, SchMurIICRA12, PalFieIICRA12}. 



Safe cooperative transportation of possibly large or bulky payloads is extremely important in various missions, such as military operations, search and rescue, mars surface explorations and personal assistance.
However, these results are based on the common and restrictive assumption that the cable connecting the payload to the quadrotor UAV is always taut and rigid. Also, the dynamic of the cable and payload are ignored and they are considered as bounded disturbances to the transporting vehicle. Therefore, they cannot be applied to aggressive, rapid load transportations where the cable is deformed or the tension along the cable is low, thereby restricting its applicability. As such, it is impossible to guarantee safety operations.
It is challenging to incorporate the effects of a deformable cable, since the dimension of the configuration space becomes infinite. Finite element approximation of a cable often yields complicated equations of motion that make dynamic analysis and controller design extremely difficult.

Recently, a coordinate-free form of the equations of motion for a chain pendulum connected a cart that moves on a horizontal plane is presented according to Lagrangian mechanics on a manifold~\cite{LeeLeoPICDC12}. This paper is an extension of the prior work of the authors in~\cite{FarhadDTLee}. By following the similar approach, in this paper, the cable is modeled as an arbitrary number of links with different sizes and masses that are serially-connected by spherical joints, as illustrated at Figure \ref{fig:Quad}. The resulting configuration manifold is the product of the special Euclidean group for the position and the attitude of the quadrotor, and a number of two-spheres that describe the direction of each link. We present Euler-Lagrange equations of the presented quadrotor model that are globally defined on the nonlinear configuration manifold.




\begin{figure}
\centerline{
	\setlength{\unitlength}{0.09\columnwidth}\scriptsize
\begin{picture}(5,7.5)(0,0)
\put(0,0){\includegraphics[width=0.45\columnwidth]{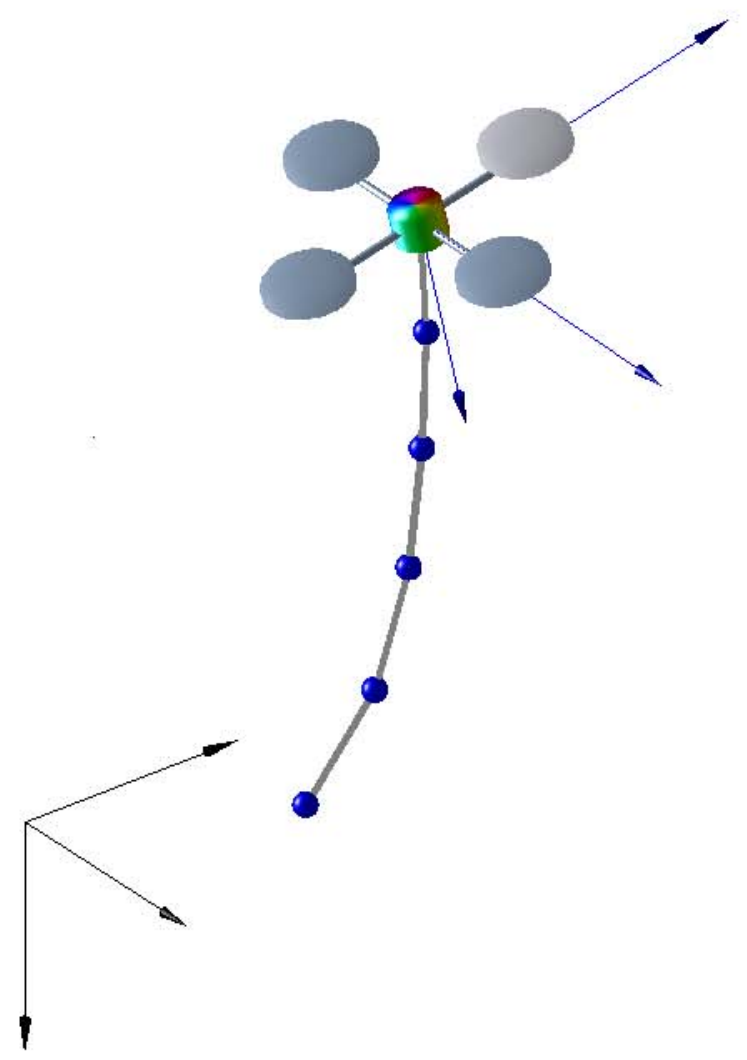}}
\put(2.7,6){\shortstack[c]{$m$}}
\put(2.3,4.8){\shortstack[c]{$m_{1}$}}
\put(2.2,4.05){\shortstack[c]{$m_{2}$}}
\put(2.1,3.25){\shortstack[c]{$m_{3}$}}
\put(2.7,2.4){\shortstack[c]{$m_{4}$}}
\put(2.2,1.5){\shortstack[c]{$m_{5}$}}
\put(1.3,0.8){\shortstack[c]{$e_{2}$}}
\put(1.7,2.1){\shortstack[c]{$e_{1}$}}
\put(0.1,-0.1){\shortstack[c]{$e_{3}$}}
\put(5,7){\shortstack[c]{$b_{1}$}}
\put(4.5,4.4){\shortstack[c]{$b_{2}$}}
\put(3.3,3.3){\shortstack[c]{$b_{3}$}}
\end{picture}
}
\caption{Quadrotor UAV with a cable-suspended load. Cable is modeled as a serial connection of arbitrary number of links (only 5 are illustrated).}\label{fig:Quad}
\end{figure}
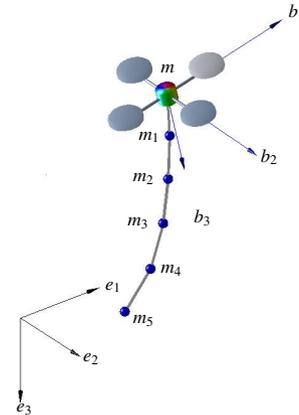

The second part of this paper deals with nonlinear control system development. Quadrotor UAV is under-actuated as the direction of the total thrust is always fixed relative to its body. By utilizing geometric control systems for quadrotor~\cite{LeeLeoPICDC10,LeeLeoAJC13,GooLeePECC13}, we show that the hanging equilibrium of the links can be asymptotically stabilized while translating the quadrotor to a desired position. In contrast to existing papers where the force and the moment exerted by the payload to the quadrotor are considered as disturbances, the control systems proposed in this paper explicitly consider the coupling effects between the cable/load dynamics and the quadrotor dynamics.

Another distinct feature is that the equations of motion and the control systems are developed directly on the nonlinear configuration manifold in a coordinate-free fashion. This yields remarkably compact expressions for the dynamic model and controllers, compared with local coordinates that often require symbolic computational tools due to complexity of multibody systems. Furthermore, singularities of local parameterization are completely avoided to generate agile maneuvers in a uniform way.

Compared with preliminary results in~\cite{GooLeePACC14}, this paper presents a rigorous Lyapunov stability analysis to establish stability properties without any timescale separation assumptions or singular perturbation, and a new nonlinear integral control term is designed to guarantee robustness against unstructured uncertainties in both rotational and translational dynamics. In short, the main contribution of this paper is presenting a nonlinear dynamic model and a control system for a quadrotor UAV with a cable-suspended load, that explicitly incorporate the effects of deformable cable. 


This paper is organized as follows. A dynamic model is presented at Section 2 and control systems are developed at Sections 3 and 4. The desirable properties of the proposed control system are illustrated by a numerical example at Section 5, followed by experimental results at Section 6.



\section{Dynamic Model of a Quadrotor with a Flexible Cable}\label{sec:DM}


Consider a quadrotor UAV with a payload that is connected via a chain of $n$ links, as illustrated at Figure \ref{fig:Quad}. The inertial frame is defined by the unit vectors $e_{1}=[1;0;0]$, $e_{2}=[0;1;0]$, and $e_{3}=[0;0;1]\in \Re^{3}$, and the third axis $e_{3}$ corresponds to the direction of gravity. Define a body-fixed frame $\{\vec{b}_{1},\vec{b}_{2},\vec{b}_{3}\}$ whose origin is located at the center of mass of the quadrotor, and its third axis $\vec b_3$ is aligned to the axis of symmetry. 

The location of the mass center, and the attitude of the quadrotor are denoted by $x\in\Re^3$ and $R\in\SO$, respectively, where the special orthogonal group is $\SO=\{R\in\Re^{3\times 3}\,|\, R^T R=I_{3\times 3},\;\mathrm{det}[R]=1\}$. A rotation matrix represents the linear transformation of a representation of a vector from the body-fixed frame to the inertial frame. 

The dynamic model of the quadrotor is identical to~\cite{LeeLeoPICDC10}. The mass and the inertia matrix of the quadrotor are denoted by $m\in\Re$ and $J\in\Re^{3\times 3}$, respectively. The quadrotor can generates a thrust $-fRe_3\in\Re^3$ with respect to the inertial frame, where $f\in\Re$ is the total thrust magnitude. It also generates a moment $M\in\Re^3$ with respect to its body-fixed frame. The pair $(f,M)$ is considered as control input of the quadrotor. 

Let $q_i\in\Sph^2$ be the unit-vector representing the direction of the $i$-th link, measured outward from the quadrotor toward the payload, where the two-sphere is the manifold of unit-vectors in $\Re^3$, i.e., $\Sph^2=\{q\in\Re^3\,|\, \|q\|=1\}$. For simplicity, we assume that the mass of each link is concentrated at the outboard end of the link, and the point where the first link is attached to the quadrotor corresponds to the mass center of the quadrotor. The mass and the length of the $i$-th link are defined by $m_i$ and $l_i\in\Re$, respectively. Thus, the mass of the payload corresponds to $m_n$. The corresponding configuration manifold of this system is given by $\SO\times\Re^3\times (\Sph^2)^n$.

Next, we show the kinematics equations. Let $\Omega\in\Re^3$ be the angular velocity of the quadrotor represented with respect to the body fixed frame, and let $\omega_i\in\Re^3$ be the angular velocity of the $i$-th link represented with respect to the inertial frame. The angular velocity is normal to the direction of the link, i.e., $q_i\cdot\omega_i=0$. The kinematics equations are given by
\begin{align}
\dot R & = R\hat\Omega,\label{eqn:Rdot}\\
\dot{q}_{i} & =\omega_{i}\times q_{i}=\hat{\omega}_{i}q_{i},\label{eqn:qidot}
\end{align}
where the hat map $\hat\cdot:\Re^3\rightarrow\so$ is defined by the condition that $\hat x y =x\times y$ for any $x,y\in\Re^3$, and it transforms a vector in $\Re^{3}$ to a $3\times 3$ skew-symmetric matrix. More explicitly, it is given by
\begin{align}\label{LinEOM}
\hat{x}=
\begin{bmatrix}
0&-x_{3}&x_{2}\\
x_{3}&0&-x_{1}\\
-x_{2}&x_{1}&0
\end{bmatrix},
\end{align}
for $x=[x_{1},x_{2},x_{3}]^{T}\in \Re^{3}$. The inverse of the hat map is denoted by the \textit{vee} map $\vee:\so\rightarrow\Re^3$. 

Throughout this paper, the 2-norm of a matrix $A$ is denoted by $\|A\|$, and the dot product is denoted by $x \cdot y = x^Ty$. Also $\lambda_{\min}(\cdot)$ and $\lambda_{\max}(\cdot)$ denotes the minimum and maximum eigenvalue of a square matrix respectively, and $\lambda_{m}$ and $\lambda_{M}$ are shorthand for $\lambda_{m}=\lambda_{m}(J)$ and $\lambda_{M}=\lambda_{M}(J)$.   

\subsection{Lagrangian}

We derive the equations of motion according to Lagrangian mechanics. The kinetic energy of the quadrotor is given by
\begin{align}
T_Q = \frac{1}{2}m\|\dot x\|^2 + \frac{1}{2} \Omega\cdot J\Omega.\label{eqn:TQ}
\end{align}
Let $x_i\in\Re^3$ be the location of $m_i$ in the inertial frame. It can be written as
\begin{align}\label{posvec33}
x_{i}=x+\sum^{i}_{a=1}{l_{a}q_{a}}.
\end{align}
Then, the kinetic energy of the links are given by
\begin{align}
T_L & = \frac{1}{2} \sum_{i=1}^n m_i \|\dot x+\sum^{i}_{a=1}{l_{a}\dot q_{a}}\|^2\nonumber\\
& = \frac{1}{2}\sum_{i=1}^n m_i \|\dot x\| + \dot x\cdot \sum_{i=1}^n\sum_{a=i}^n m_a l_i \dot q_i
+\frac{1}{2}\sum_{i=1}^n m_i \|\sum_{a=1}^i l_a \dot q_a\|^2.
\label{eqn:TL}
\end{align}
From \refeqn{TQ} and \refeqn{TL}, the total kinetic energy can be written as
\begin{align}
T & =\frac{1}{2}M_{00}\|\dot{x}\|^{2}+\dot{x}\cdot\sum^{n}_{i=1}{M_{0i}\dot{q}_{i}}+\frac{1}{2}\sum^{n}_{i,j=1}{M_{ij}\dot{q}_{i}\cdot\dot{q}_{j}}\nonumber\\
&\quad +\frac{1}{2}\Omega^{T}J\Omega,\label{eqn:KE}
\end{align}
where the inertia values $M_{00},M_{0i},M_{ij}\in\Re$ are given by
\begin{gather}
M_{00}=m+\sum_{i=1}^n m_i,\quad M_{0i}=\sum_{a=i}^n m_a l_i,\quad M_{i0}=M_{0i},\nonumber\\
M_{ij}=\braces{\sum_{a=\max\{i,j\}}^n m_a} l_i l_j,\label{eqn:Mij}
\end{gather}
for $1\leq i,j\leq n$. The gravitational potential energy is given by
\begin{align}
V & = -mgx\cdot e_3 - \sum_{i=1}^n m_i g x_i\cdot e_3\nonumber\\
& = -\sum^{n}_{i=1}\sum^{n}_{a=i}m_{a}gl_{i}e_{3}\cdot q_{i}-M_{00}ge_{3}\cdot x,\label{eqn:PE}
\end{align}
From \refeqn{KE} and \refeqn{PE}, the Lagrangian is $L=T-V$.

\subsection{Euler-Lagrange equations}
Coordinate-free form of Lagrangian mechanics on the two-sphere $\Sph^2$ and the special orthogonal group $\SO$ for various multibody systems has been studied in~\cite{Lee08,LeeLeoIJNME08}. The key idea is representing the infinitesimal variation of $R\in\SO$ in terms of the exponential map
\begin{align}
\delta R = \frac{d}{d\epsilon}\bigg|_{\epsilon = 0} \exp R(\epsilon \hat\eta) = R\hat\eta,\label{eqn:delR}
\end{align}
for $\eta\in\Re^3$. The corresponding variation of the angular velocity is given by $\delta\Omega=\dot\eta+\Omega\times\eta$. Similarly, the infinitesimal variation of $q_i\in\Sph^2$ is given by
\begin{align}
\delta q_i = \xi_i\times q_i,\label{eqn:delqi}
\end{align}
for $\xi_i\in\Re^3$ satisfying $\xi_i\cdot q_i=0$. This lies in the tangent space as it is perpendicular to $q_{i}$. Using these, we obtain the following Euler-Lagrange equations.
\begin{prop}\label{prop:FDM}
Consider a quadrotor with a cable suspended payload whose Lagrangian is given by \refeqn{KE} and \refeqn{PE}. The Euler-Lagrange equations on $\Re^3\times\SO\times(\Sph^2)^n$ are as follows
\begin{gather}
M_{00}\ddot{x}+\sum^{n}_{i=1}{M_{0i}\ddot{q}_{i}}=-fRe_{3}+M_{00}ge_{3}+\Delta_{x},\label{eqn:xddot}\\
M_{ii}\ddot q_i  -\hat q_i^2 (M_{i0}\ddot x + \sum_{\substack{j=1\\j\neq i}}^n M_{ij}\ddot q_j)\nonumber\\
=- M_{ii}\|\dot q_i\|^2 q_i-\sum_{a=i}^n m_a gl_i\hat q_i^2 e_3,\label{eqn:qddot}\\
J\dot{\Omega}+\hat{\Omega}J\Omega=M+\Delta_{R},\label{eqn:Wdot}
\end{gather}
where $M_{ij}$ is defined at \refeqn{Mij}. Therefore $\Delta_{x}$ and $\Delta_{R}\in\Re^3$ are fixed disturbances applied to the translational and rotational dynamics of the quadrotor respectively. Equations \refeqn{xddot} and \refeqn{qddot} can be rewritten in a matrix form as follows:
\begin{align}
&\begin{bmatrix}%
    M_{00} & M_{01} & M_{02} & \cdots & M_{0n} \\
    -\hat q_1^2 M_{10} & M_{11}I_{3} & -M_{12} \hat q_1^2 & \cdots & -M_{1n}\hat q_1^2\\%
    -\hat q_2^2 M_{20} & -M_{21} \hat q_2^2 & M_{22} I_{3} & \cdots & -M_{2n} \hat q_2^2\\%
    \vdots & \vdots & \vdots & & \vdots\\
    -\hat q_n^2 M_{n0} & -M_{n1} \hat q_n^2 & -M_{n2}\hat q_n^2 & \cdots & M_{nn} I_{3}
    \end{bmatrix}%
    \begin{bmatrix}
    \ddot x \\ \ddot q_1 \\ \ddot q_2 \\ \vdots \\ \ddot q_n
    \end{bmatrix}\nonumber\\
 &=   \begin{bmatrix}
    -fRe_3 +M_{00}ge_3+\Delta_{x}\\
    -\|\dot q_1\|^2M_{11} q_1 -\sum_{a=1}^n m_a gl_1\hat q_1^2 e_3\\
    -\|\dot q_2\|^2M_{22} q_2 -\sum_{a=2}^n m_a gl_2\hat q_2^2 e_3\\
    \vdots\\
    -\|\dot q_n\|^2M_{nn} q_n - m_n gl_n\hat q_n^2 e_3
    \end{bmatrix}.\label{eqn:ELm}
\end{align}
Or equivalently, it can be written in terms of the angular velocities as
\begin{gather}
\begin{bmatrix}%
    M_{00} & -M_{01}\hat q_1 & -M_{02}\hat q_2 & \cdots & -M_{0n}\hat q_n\\
    \hat q_1 M_{10} & M_{11}I_{3} & -M_{12} \hat q_1 \hat q_2 & \cdots & -M_{1n}\hat q_1 \hat q_n\\%
    \hat q_2 M_{20} &-M_{21} \hat q_2\hat q_1 & M_{22} I_{3} & \cdots & -M_{2n} \hat q_2 \hat q_n\\%
    \vdots & \vdots & \vdots & & \vdots\\
    \hat q_n M_{n0} &-M_{n1} \hat q_n \hat q_1 & -M_{n2}\hat q_n \hat q_2 & \cdots & M_{nn} I_{3}
    \end{bmatrix}%
    \begin{bmatrix}
    \ddot x \\ \dot \omega_1 \\ \dot \omega_2 \\ \vdots \\ \dot \omega_n
    \end{bmatrix}\nonumber\\
    =
    \begin{bmatrix}
    \sum_{j=1}^n M_{0j} \|\omega_j\|^2 q_j-fRe_3+M_{00}ge_3+\Delta_{x}\\
    \sum_{j=2}^n M_{1j}\|\omega_j\|^2\hat q_1 q_j +\sum_{a=1}^n m_a gl_1\hat q_1 e_3\\
    \sum_{j=1,j\neq 2}^n M_{2j}\|\omega_j\|^2\hat q_2 q_j +\sum_{a=2}^n m_a gl_2\hat q_2 e_3\\
    \vdots\\
    \sum_{j=1}^{n-1} M_{nj}\|\omega_j\|^2\hat q_n q_j + m_n gl_n\hat q_n e_3\\
    \end{bmatrix},\label{eqn:ELwm}\\
\dot q_i = \omega_i\times q_i.\label{eqn:ELwm2}
\end{gather}
\end{prop}
\begin{proof}
See Appendix \ref{sec:PfFDM}
\end{proof}
These provide a coordinate-free form of the equations of motion for the presented quadrotor UAV that is uniformly defined for any number of links $n$, and that is globally defined on the nonlinear configuration manifold. Compared with equations of motion derived in terms of local coordinates, such as Euler-angles, these avoid singularities completely, and they provide a compact form of equations that are suitable for control system design.

\section{Control System Design for a Simplified Dynamic Model}\label{sec:SDM}

\subsection{Control Problem Formulation}

Let $x_d\in\Re^3$ be a fixed desired location of the quadrotor UAV. Assuming that all of the links are pointing downward, i.e., $q_i=e_3$, the resulting location of the payload is given by 
\begin{align}
x_n=x_d +\sum_{i=1}^n l_i e_3. 
\end{align}
We wish to design the control force $f$ and the control moment $M$ such that this hanging equilibrium configuration at the desired location becomes asymptotically stable. 

\subsection{Simplified Dynamic Model}

For the given equations of motion \refeqn{xddot} for $x$, the control force is given by $-fRe_3$. This implies that the total thrust magnitude $f$ can be arbitrarily chosen, but the direction of the thrust vector is always along the third body-fixed axis. Also, the rotational attitude dynamics of the quadrotor is not affected by the translational dynamics of the quadrotor or the dynamics of links.

To overcome the under-actuated property of a quadrotor, in this section, we first replace the term  $-fRe_3$ of \refeqn{xddot} by a fictitious control input $u\in\Re^3$, and design an expression for $u$ to asymptotically stabilize the desired equilibrium. This is equivalent to assuming that the attitude $R$ of the quadrotor can be instantaneously controlled. The effects of the attitude dynamics are incorporated at the next section. Also $\Delta_{x}$ is ignored in the simplified dynamic model.
In short, the equations of motion for the simplified dynamic model considered in the section are given by
\begin{align}
M_{00}\ddot{x}+\sum^{n}_{i=1}{M_{0i}\ddot{q}_{i}}=u+M_{00}ge_{3},\label{eqn:xddot_sim}
\end{align}
and \refeqn{qddot}.

\subsection{Linear Control System}\label{sec:LCS}
The fictitious control input is designed from the linearized dynamics about the desired hanging equilibrium. The variation of $x$ and $u$ are given by
\begin{align}
\delta x = x - x_d,\quad \delta u = u - M_{00}g e_3.\label{eqn:delxLin}
\end{align}
From \refeqn{delqi}, the variation of $q_i$ from the equilibrium can be written as
\begin{align}
\delta q_i = \xi_i\times e_3,\label{eqn:delqLin}
\end{align}
where $\xi_i\in\Re^3$ with $\xi_i\cdot e_3=0$. The variation of $\omega_i$ is given by $\delta\omega\in\Re^3$ with $\delta\omega_i \cdot e_3=0$. Therefore, the third element of each of $\xi_i$ and $\delta\omega_i$ for any equilibrium configuration is zero, and they are omitted in the following linearized equation, i.e., the state vector of the linearized equation is composed of $C^T\xi_i\in\Re^2$, where $C=[e_1,e_2]\in\Re^{3\times 2}$. 

\newcommand{\Mb}{\mathbf{M}}
\newcommand{\Kb}{\mathbf{K}}
\newcommand{\Bb}{\mathbf{B}}
\newcommand{\xb}{\mathbf{x}}
\newcommand{\ub}{\mathbf{u}}
\newcommand{\vb}{\mathbf{v}}
\newcommand{\Gb}{\mathbf{G}}

\begin{prop}\label{prop:FDMM}
The linearized equations of the simplified dynamic model \refeqn{xddot_sim} and \refeqn{qddot} can be written as follows
\begin{gather}
\Mb\ddot \xb  + \Gb\xb = \Bb \delta u+ \g(\xb,\dot{\xb}),\label{eqn:Lin}
\end{gather}
where $\g(\xb,\dot{\xb})$ corresponds to the higher order terms where $\xb=[\delta x,\; \xb_{q}]^{T}\in\Re^{2n+3}$, $\Mb\in\Re^{2n+3\times 2n+3}$, $\Gb\in\Re^{2n+3\times 2n+3}$, $\Bb\in\Re^{2n+3\times 3}$, and \refeqn{Lin} can equivalently be written as
\begin{align*}
\begin{bmatrix} \Mb_{xx} & \Mb_{xq}\\ \Mb_{qx} & \Mb_{qq} \end{bmatrix}
\begin{bmatrix}  \delta\ddot x \\ \ddot \xb_q\end{bmatrix}
&+
\begin{bmatrix} 0_3 & 0_{3\times 2n}\\ 0_{2n\times 3} & \Gb_{qq}\end{bmatrix}
\begin{bmatrix}  \delta x \\ \xb_q\end{bmatrix}\\
&=
\begin{bmatrix} I_3 \\ 0_{2n\times 3}\end{bmatrix}
\delta u+ \g(\xb,\dot{\xb}),
\end{align*}
where the corresponding sub-matrices are defined as
\begin{align*}
\xb_q & = [C^T \xi_1;\,\ldots\,;\,C^T \xi_n],\\
\Mb_{xx} &= M_{00}I_{3},\\
\Mb_{xq} &= \begin{bmatrix}
-M_{01}\hat e_3C & -M_{02}\hat e_3C & \cdots & -M_{0n}\hat e_3C
\end{bmatrix},\\
\Mb_{qx} & = \Mb_{xq}^T,\\
\Mb_{qq} &=
    \begin{bmatrix}%
M_{11}I_{2} & M_{12} I_2 & \cdots & M_{1n}I_2\\%
M_{21} I_2 & M_{22} I_{2} & \cdots & M_{2n}I_2\\%
\vdots & \vdots & & \vdots\\
M_{n1}I_2 & M_{n2}I_2 & \cdots & M_{nn} I_{2}
    \end{bmatrix},\\
\Gb_{qq}  = \mathrm{diag}[&\sum_{a=1}^n {m_a gl_1 I_2},\cdots,m_ngl_nI_2].
\end{align*}
\end{prop}

\begin{proof}
See Appendix \ref{sec:PfFDMM}
\end{proof}
For the linearized dynamics \refeqn{Lin}, the following control system is chosen
\begin{align}
\delta u & = -k_{x}\delta{x}-k_{\dot{x}}\delta\dot{x}-\sum_{a=1}^{n}{k_{q_{i}}C^{T}(e_3\times q_{i})}-k_{\omega_{i}}C^{T}\delta\omega_{i}\nonumber\\
& = -K_x \xb - K_{\dot x} \dot \xb,\label{eqn:delu}
\end{align}
for controller gains $K_x =[k_xI_3,k_{q_1}I_{3\times 2},\ldots k_{q_n}I_{3\times 2}]\in\Re^{3\times (3+2n)}$ and $K_{\dot x} =[k_{\dot x}I_3,k_{\omega_1}I_{3\times 2},\ldots k_{\omega_n}I_{3\times 2}]\in\Re^{3\times (3+2n)}$. Provided that \refeqn{Lin} is controllable, we can choose the controller gains $K_x,K_{\dot x}$ such that the equilibrium is asymptotically stable for the linearized equation \refeqn{Lin}. Then, the equilibrium becomes asymptotically stable for the nonlinear Euler Lagrange equation \refeqn{xddot_sim} and \refeqn{qddot}~\cite{Kha96}. The controlled linearized system can be written as
\begin{align}
\dot{z}_{1}=&\mathds{A} z_{1}+\mathds{B}\g(\xb,\dot{\xb}),
\end{align}
where $z_{1}=[\xb,\; \dot{\xb}]^{T}\in\Re^{4n+6}$ and the matrices $\mathds{A}\in\Re^{4n+6\times 4n+6}$ and $\mathds{B}\in\Re^{4n+6\times 2n+3}$ are defined as
\begin{align}
\mathds{A}=\begin{bmatrix}
0&I\\
-\Mb^{-1}(\Gb+\Bb K_{x})&-{(\Mb^{-1}\Bb K_{\dot{x}})}
\end{bmatrix}, \mathds{B}=\begin{bmatrix}
0\\
\Mb^{-1}
\end{bmatrix}.
\end{align}
We can also choose $K_{\xb}$ and $K_{\dot{\xb}}$ such that $\mathds{A}$ is Hurwitz. Then for any positive definite matrix $Q\in\Re^{4n+6\times4n+6}$, there exist a positive definite and symmetric matrix $P\in\Re^{4n+6\times4n+6}$ such that $\mathds{A}^{T}P+P\mathds{A}=-Q$ according to~\cite[Thm 3.6]{Kha96}.

\section{Controller Design for a Quadrotor with a Flexible Cable}\label{sec:CS}

The control system designed in the previous section is generalized to the full dynamic model that includes the attitude dynamics. The central idea is that the attitude $R$ of the quadrotor is controlled such that its total thrust direction $-Re_3$ that corresponds to the third body-fixed axis asymptotically follows the direction of the fictitious control input $u$. By choosing the total thrust magnitude properly, we can guarantee that the total thrust vector $-fRe_{3}$ asymptotically converges to the fictitious ideal force $u$, thereby yielding asymptotic stability of the full dynamic model.
\subsection{Controller Design}
Consider the full nonlinear equations of motion, let $A\in\Re^3$ be the ideal total thrust of the quadrotor system that asymptotically stabilize the desired equilibrium. From \refeqn{delxLin}, we have 
\begin{align}
A= M_{00}ge_3 + \delta u = -K_{x} \xb-K_{\dot{x}}\dot\xb -K_{z}\satr_{\sigma}(e_{\xb})+ M_{00}ge_3,\label{eqn:A}
\end{align}
where the following integral term $e_{\xb}\in\Re^{2n+3}$ is added to eliminate the effect of disturbance $\Delta_x$ in the full dynamic model 
\begin{align}\label{eqn:exterm}
e_{\xb}=\int^{t}_{0}{(P\mathds{B})^{T}z_{1}(\tau)\;d\tau},
\end{align}
where $K_z =[k_{z}I_3,k_{z_1}I_{3\times 2},\ldots k_{z_n}I_{3\times 2}]\in\Re^{3\times (3+2n)}$ is an integral gain. For a positive constant $\sigma\in\Re$, a saturation function $\sat_\sigma:\Re\rightarrow [-\sigma,\sigma]$ is introduced as
\begin{align*}
\sat_{\sigma}(y) = \begin{cases}
\sigma & \mbox{if } y >\sigma\\
y & \mbox{if } -\sigma \leq y \leq\sigma\\
-\sigma & \mbox{if } y <-\sigma\\
\end{cases}.
\end{align*}
If the input is a vector $y\in\Re^n$, then the above saturation function is applied element by element to define a saturation function $\sat_\sigma(y):\Re^n\rightarrow [-\sigma,\sigma]^n$ for a vector. It is also assumed that an upper bound of the infinite norm of the uncertainty is known
\begin{align}\label{eqn:disturbancecond}
\|\Delta_{x}\|_{\infty}\leq \delta,
\end{align}
for positive constant $\delta$. The desired direction of the third body-fixed axis $b_{3_c}\in\Sph^2$ is given by
\begin{align}
b_{3_c} = - \frac{A}{\|A\|}.\label{eqn:b3c}
\end{align}
This provides a two-dimensional constraint for the desired attitude of quadrotor, and there is additional one-dimensional degree of freedom that corresponds to rotation about the third body-fixed axis, i.e., yaw angle. A desired direction of the first body-fixed axis, namely $b_{1_d}\in\Sph^2$ is introduced to resolve it, and it is projected onto the plane normal to $b_{3_c}$. The desired direction of the second body-fixed axis is chosen to constitute an orthonormal frame. More explicitly, the desired attitude is given by
\begin{align}
R_c = \bracket{-\frac{\hat b_{3_c}^2 b_{1_d}}{\|\hat b_{3_c}^2 b_{1_d}\|},\;
 \frac{\hat b_{3_c}b_{1_d}}{\|\hat b_{3_c}b_{1_d}\|},\; b_{3_c}},
\end{align}
which is guaranteed to lie in $\SO$ by construction, assuming that $b_{1_d}$ is not parallel to $b_{3_c}$~\cite{LeeLeoAJC13}. The desired angular velocity $\Omega_{c}\in\Re^{3}$ is obtained by the attitude kinematics equation
\begin{align}
\Omega_c = (R_c^T \dot R_c)^\vee.
\end{align}
Next, we introduce the tracking error variables for the attitude and the angular velocity $e_{R}$, $e_{\Omega}\in\Re^{3}$ as follows~\cite{TFJCHTLeeHG}
\begin{align}
&e_{R}=\frac{1}{2}(R_{c}^{T}R-R^{T}R_{c})^{\vee},\\
&e_{\Omega}=\Omega-R^{T}R_{c}\Omega_{c}.
\end{align}
The thrust magnitude and the moment vector of quadrotor are chosen as 
\begin{align}
f = -A\cdot Re_3,\label{eqn:fi}
\end{align}\vspace{-0.87cm}
\begin{align}
M =&-{k_R}e_{R} -{k_{\Omega}}e_{\Omega} -k_{I}e_{I}+\Omega\times J\Omega\nonumber\\
&-J(\hat\Omega R^T R_{c} \Omega_{c} - R^T R_{c}\dot\Omega_{c}),\label{eqn:Mi}
\end{align}
where $k_R,k_\Omega$, and $k_{I}$ are positive constants and the following integral term is introduced to eliminate the effect of fixed disturbance $\Delta_{R}$
\begin{align}\label{eqn:integralterm}
e_{I}=\int^{t}_{0}{e_{\Omega}(\tau)+c_{2}e_{R}(\tau)d\tau},
\end{align}
where $c_{2}$ is a positive constant.
\subsection{Stability Analysis}
\begin{prop}\label{prop:SA1}
Consider control inputs $f$, $M$ defined in \refeqn{fi} and \refeqn{Mi}. There exist controller parameters and gains such that, (i) the zero equilibrium of tracking error is stable in the sense of Lyapunov; (ii) the tracking errors $e_{R}$, $e_{\Omega}$, $\xb$, $\dot{\xb}$ asymptotically converge to zero as $t\rightarrow\infty$; (iii) the integral terms $e_{I}$ and $e_{\xb}$ are uniformly bounded.

\end{prop}
\begin{proof}
See Appendix \ref{sec:stability1}
\end{proof}
By utilizing geometric control systems for quadrotor, we show that the hanging equilibrium of the links can be asymptotically stabilized while translating the quadrotor to a desired position. The control systems proposed explicitly consider the coupling effects between the cable/load dynamics and the quadrotor dynamics. We presented a rigorous Lyapunov stability analysis to establish stability properties without any timescale separation assumptions or singular perturbation, and a new nonlinear integral control term is designed to guarantee robustness against unstructured uncertainties in both rotational and translational dynamics.


\section{Numerical Example}\label{sec:NE}

The desirable properties of the proposed control system are illustrated by a numerical example. Properties of a quadrotor are chosen as
\begin{align*}
m=0.5\,\mathrm{kg},\quad J=\mathrm{diag}[0.557,\,0.557,\,1.05]\times 10^{-2}\,\mathrm{kgm^2}.
\end{align*}
Five identical links with $n=5$, $m_i=0.1\,\mathrm{kg}$, and $l_i=0.1\,\mathrm{m}$ are considered. Controller parameters are selected as follows: $k_x=12.8$, $k_v=4.22$, ${k_R}=0.65$, ${k_\Omega}= 0.11$, $k_{I}=1.5$, $c_{1}=c_{2}=0.7$. Also $k_{q}$ and $k_{\omega}$ are defined as 
\begin{align*}
&k_q=[11.01,\,6.67,\,1.97,\,0.41,\,0.069],\\
&k_\omega=[0.93,\,0.24,\,0.032,\,0.030,\,0.025].
\end{align*}
\begin{figure}
\centerline{\hspace{-0.5cm}
	\subfigure[Attitude error function $\psi$]{
		\includegraphics[width=0.6\columnwidth]{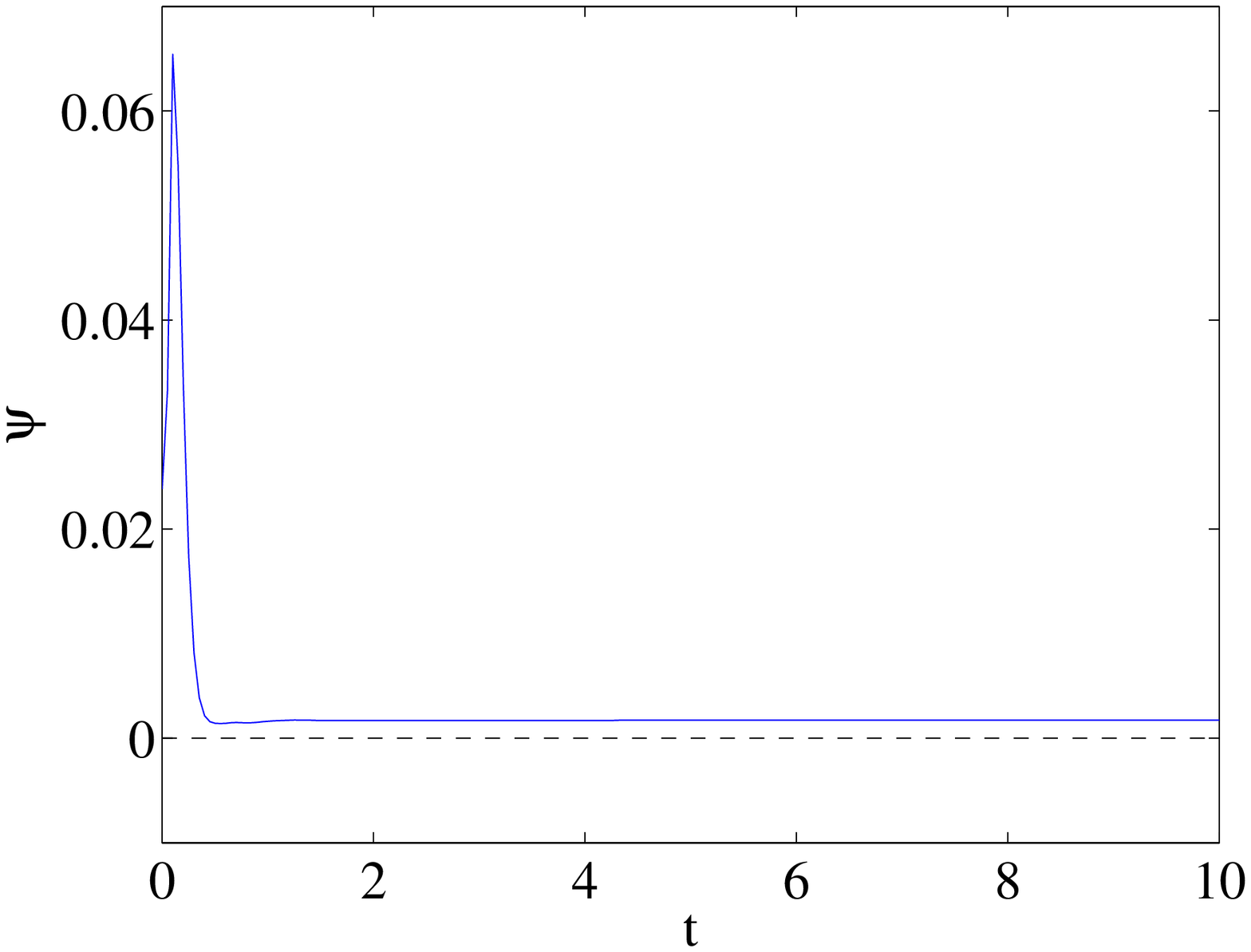}}\hspace{-0.5cm}
	\subfigure[Direction error $e_{q}$ and angular velocity error $e_{\omega}$ for links]{
		\includegraphics[width=0.6\columnwidth]{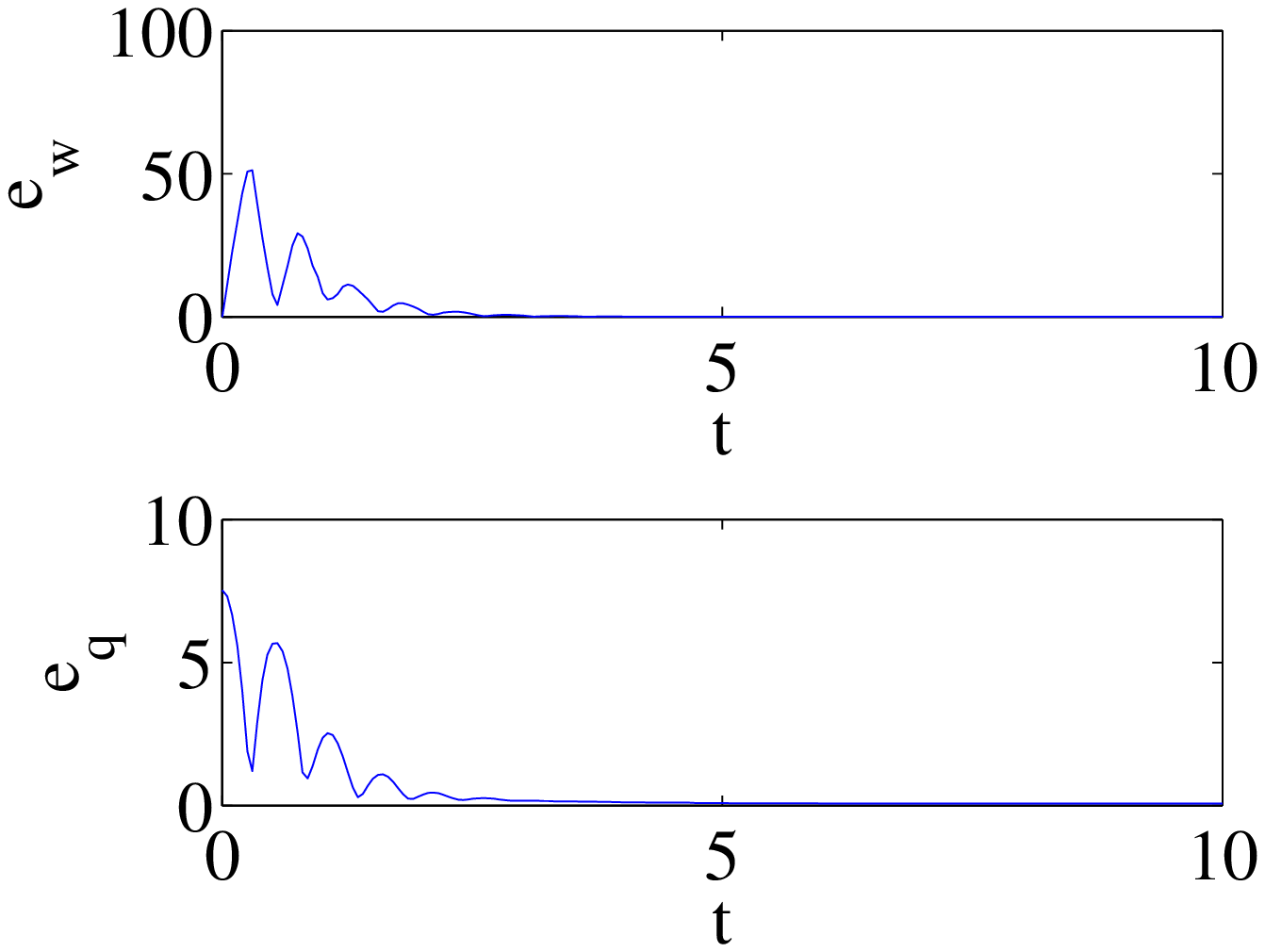}}
}
\centerline{\hspace{-0.3cm}
	\subfigure[Quadrotor angular velocity $\Omega$:blue, $\Omega_{d}$:red]{
		\includegraphics[width=0.6\columnwidth]{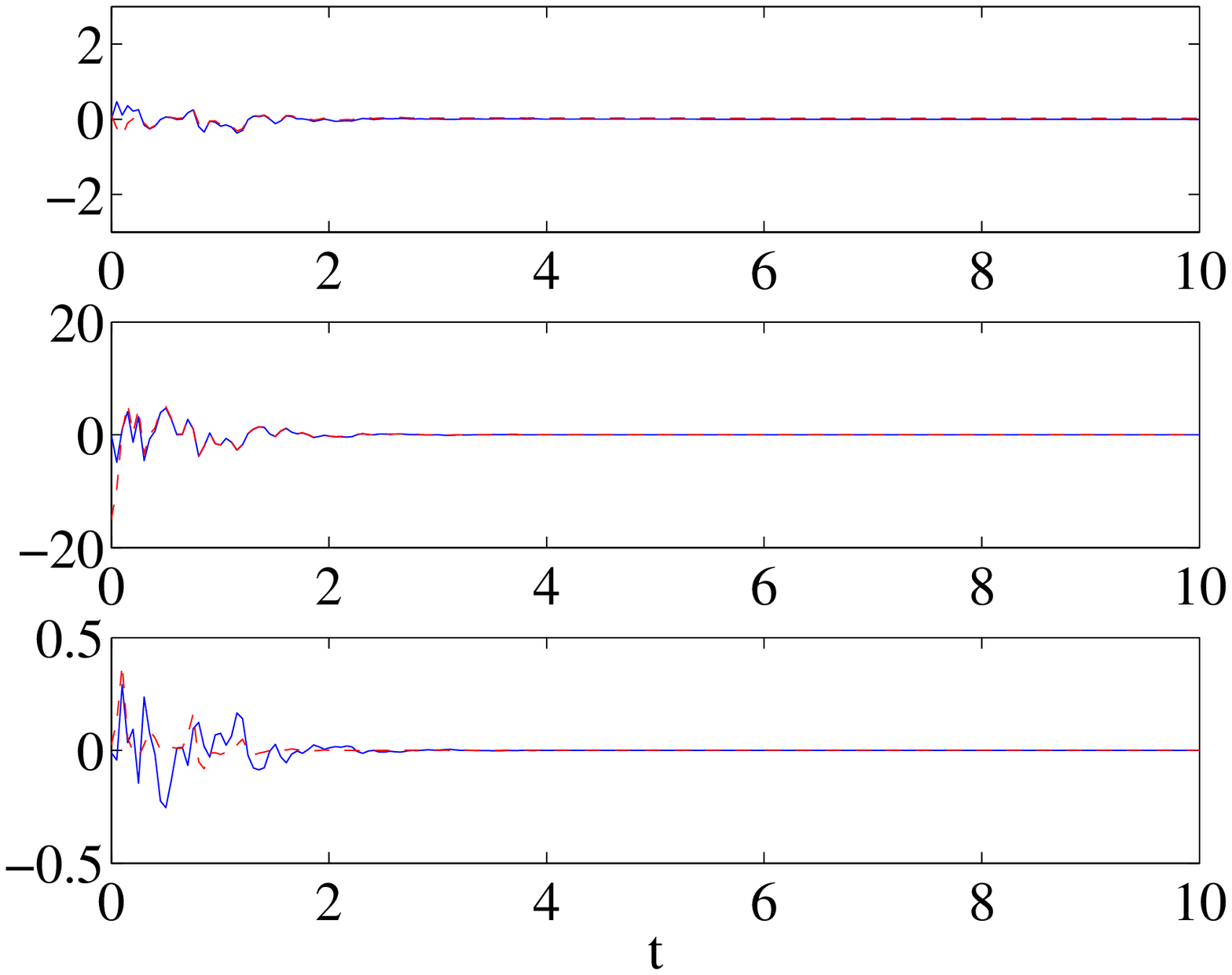}}\hspace{-0.5cm}
	\subfigure[Control force $u$]{
		\includegraphics[width=0.6\columnwidth]{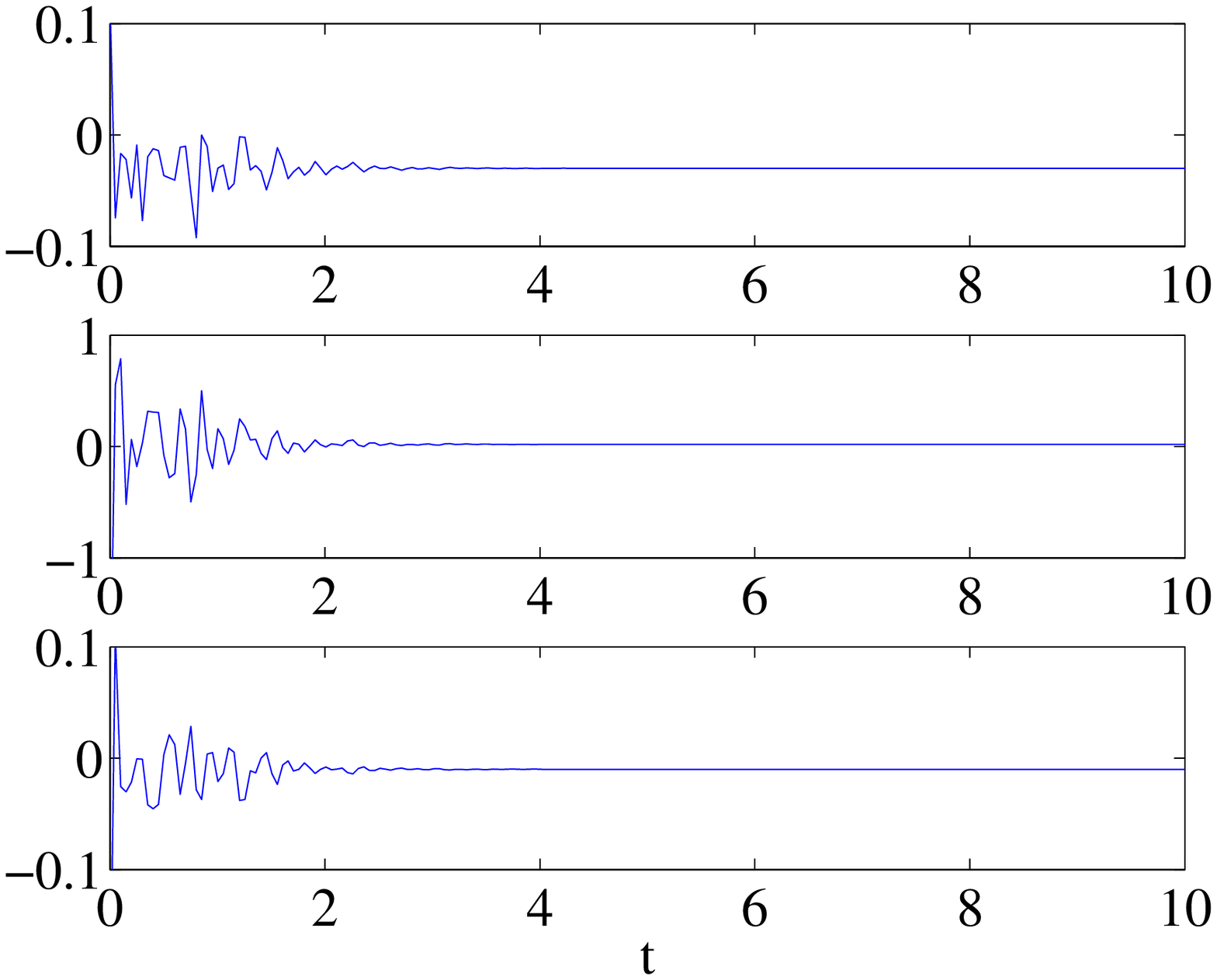}}
}
\centerline{\hspace{-0.3cm}
	\subfigure[Quadrotor position]{
		\includegraphics[width=0.6\columnwidth]{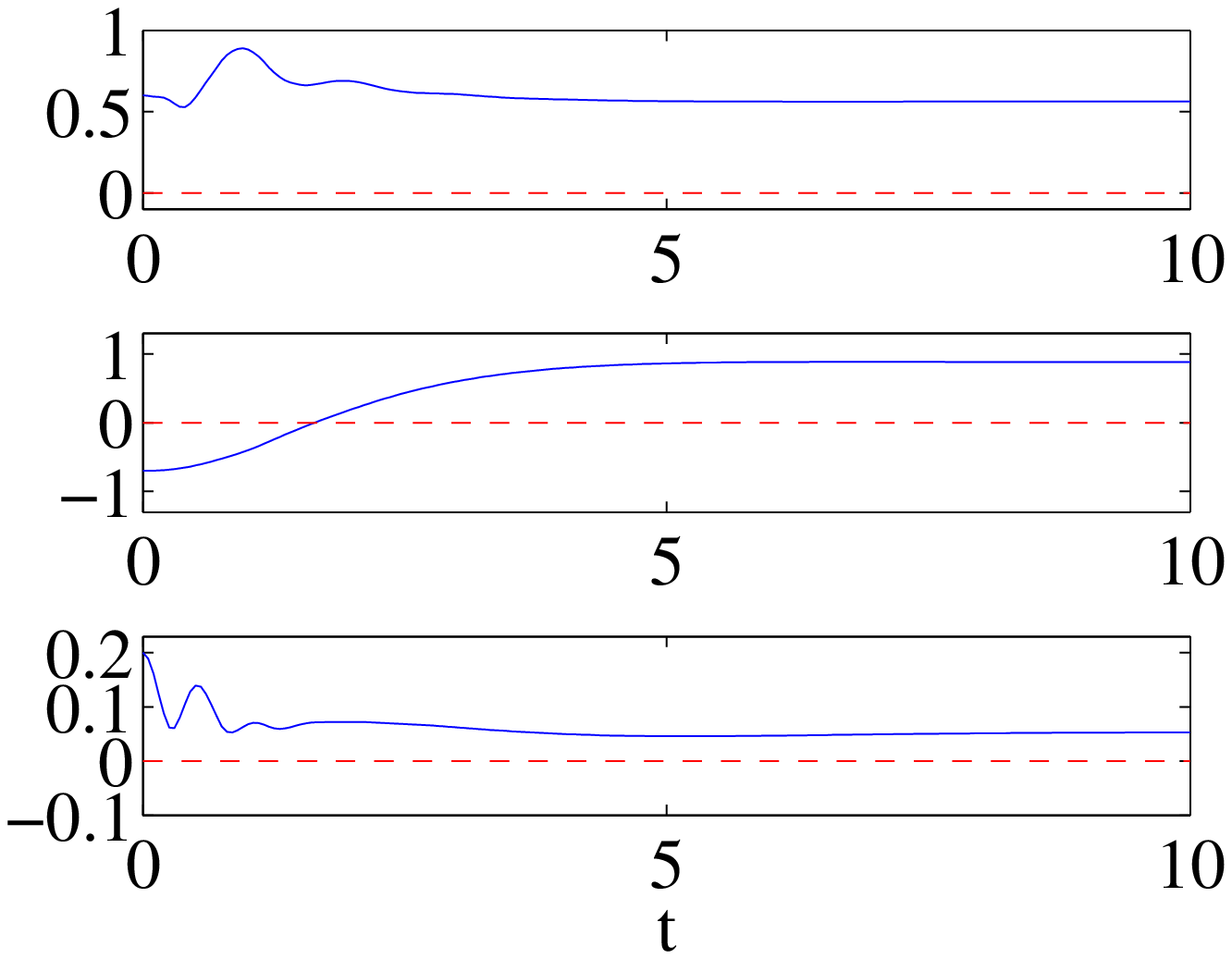}}\hspace{-0.5cm}
	\subfigure[Quadrotor velocity]{
		\includegraphics[width=0.6\columnwidth]{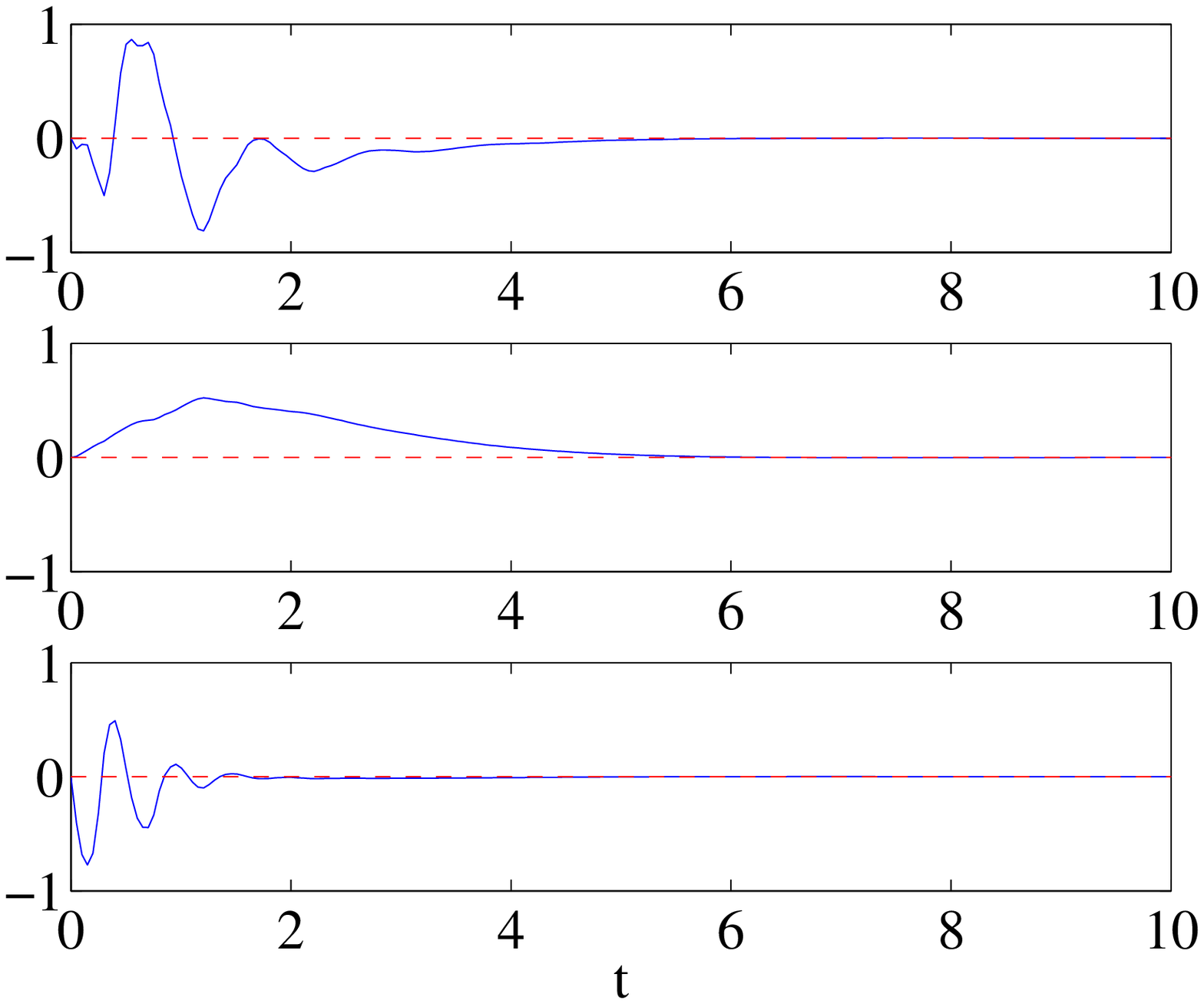}}
}
\caption{Stabilization of a payload connected to a quadrotor with 5 links without Integral term}\label{fig:simresultsWO}
\end{figure}
The desired location of the quadrotor is selected as $x_d=0_{3\times 1}$. The initial conditions for the quadrotor are given by
\begin{gather*}
x(0)=[0.6;-0.7;0.2], \ \dot{x}(0)=0_{3\times 1},\\
R(0)=I_{3\times 3},\quad \Omega(0)=0_{3\times 1}.
\end{gather*}
The initial direction of the links are chosen such that the cable is curved along the horizontal direction, as illustrated at Figure \ref{fig:fisrt_sub11}, and the initial angular velocity of each link is chosen as zero. 

The following two fixed disturbances also considered in the equations
\begin{align*}
&\Delta_R=[0.03,-0.02,0.01]^T,\\
&\Delta_x=[-0.0125,0.0125,0.01]^T.
\end{align*}
We define the two following error functions to show the stabilizing performance for the links:
\begin{align}
e_{q}=\sum_{i=1}^{n}{\|q_{i}-e_3\|},\quad e_{\omega}=\sum_{i=1}^{n}{\|\omega_{i}\|}.
\end{align}
\begin{figure}
\centerline{
	\subfigure[Attitude error function $\psi$]{\hspace{0.3cm}
		\includegraphics[width=0.6\columnwidth]{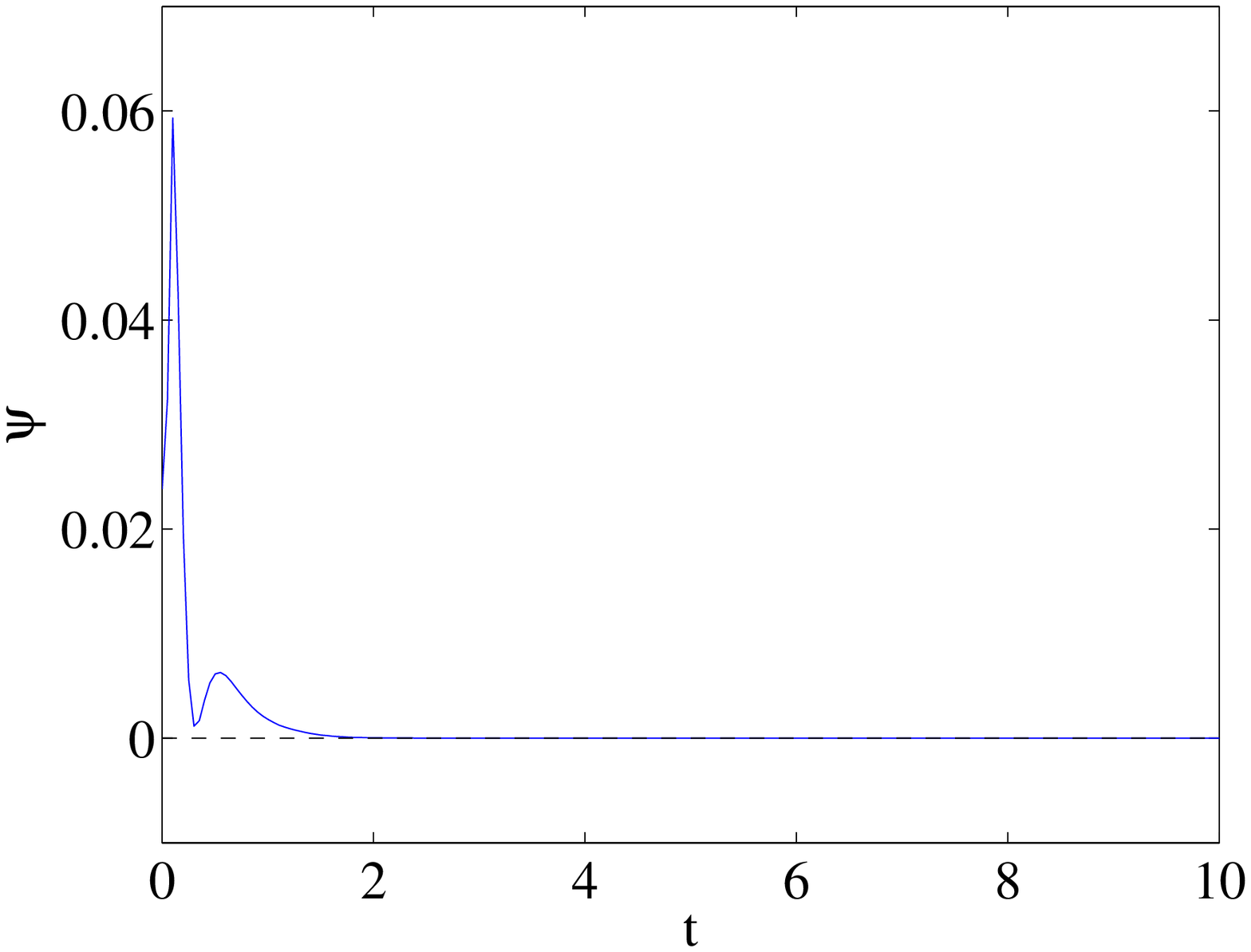}}\hspace{-0.5cm}
	\subfigure[Direction error $e_{q}$ and angular velocity error $e_{\omega}$ for links]{
		\includegraphics[width=0.6\columnwidth]{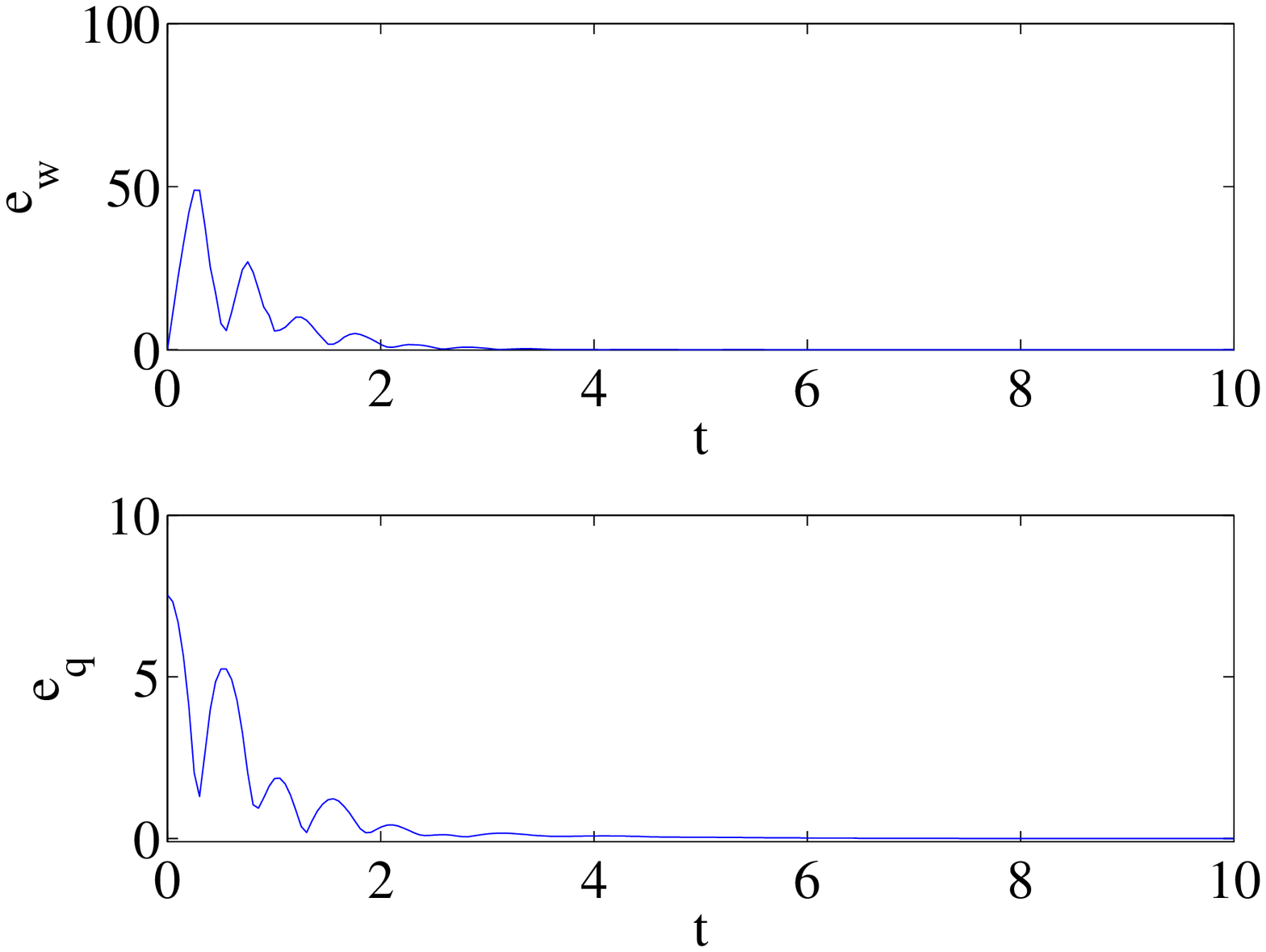}}
}
\centerline{\hspace{0.3cm}
	\subfigure[Quadrotor angular velocity $\Omega$:blue, $\Omega_{d}$:red]{
		\includegraphics[width=0.6\columnwidth]{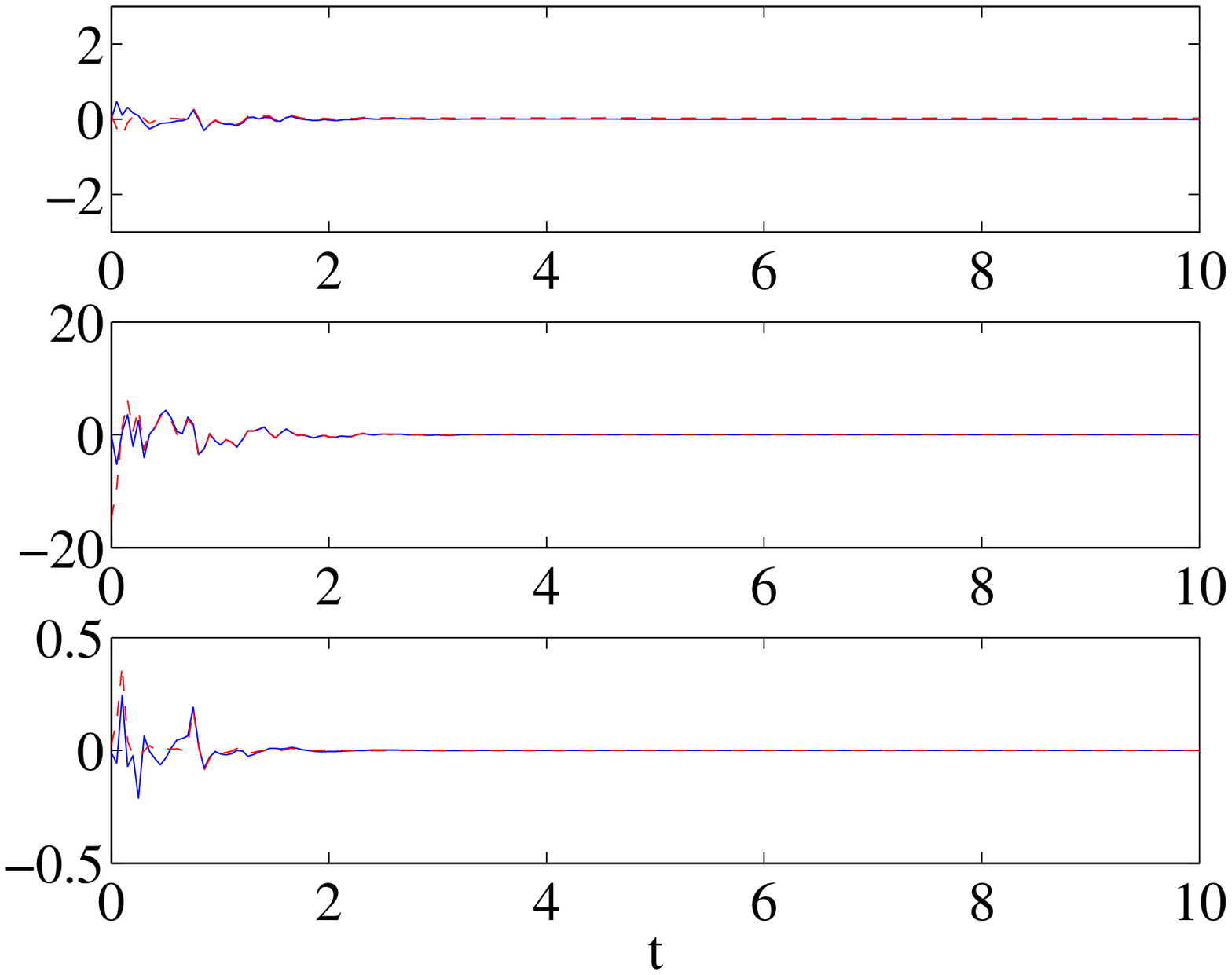}}\hspace{-0.5cm}
	\subfigure[Control force $u$]{
		\includegraphics[width=0.6\columnwidth]{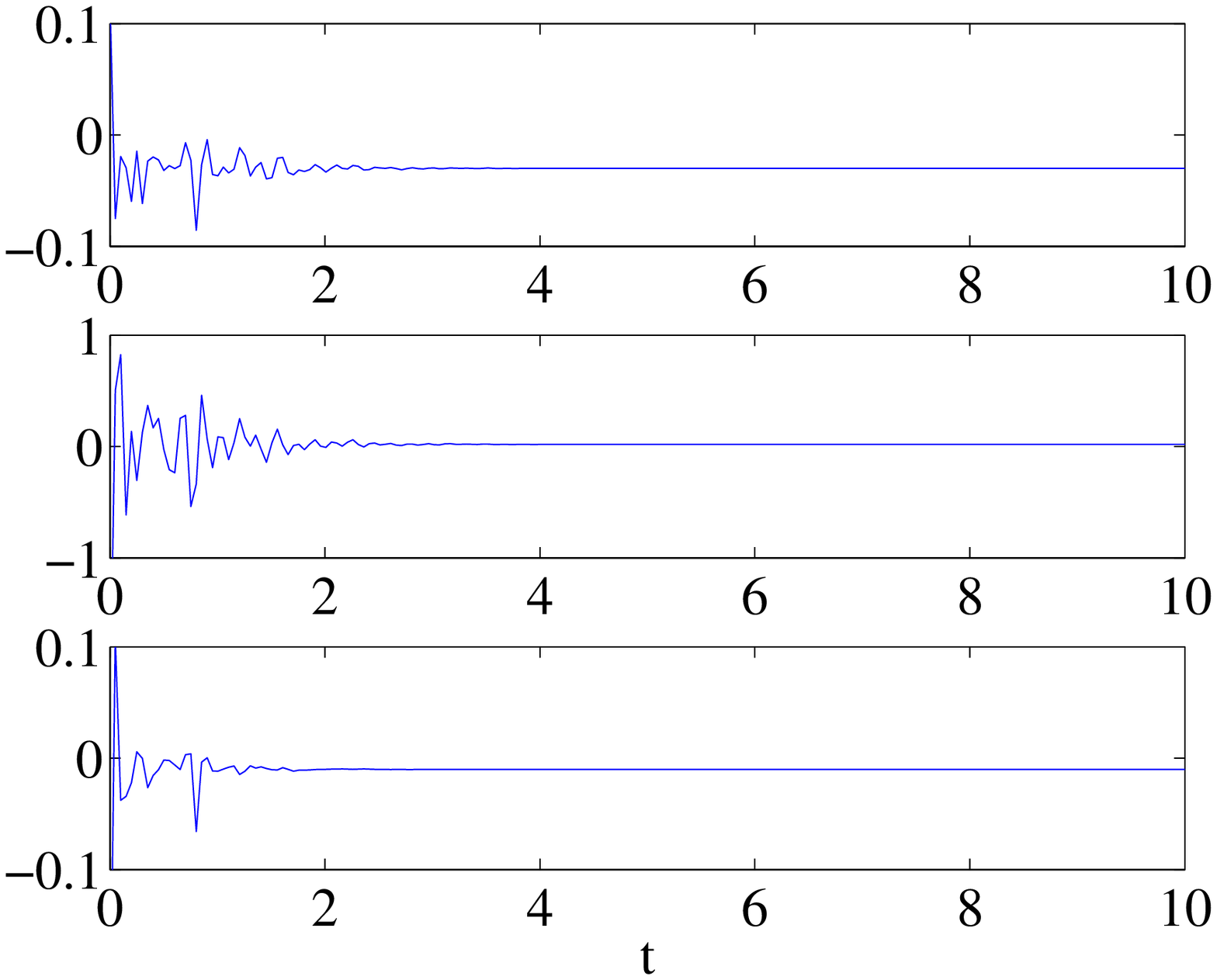}}
}
\centerline{\hspace{0.3cm}
	\subfigure[Quadrotor position]{
		\includegraphics[width=0.6\columnwidth]{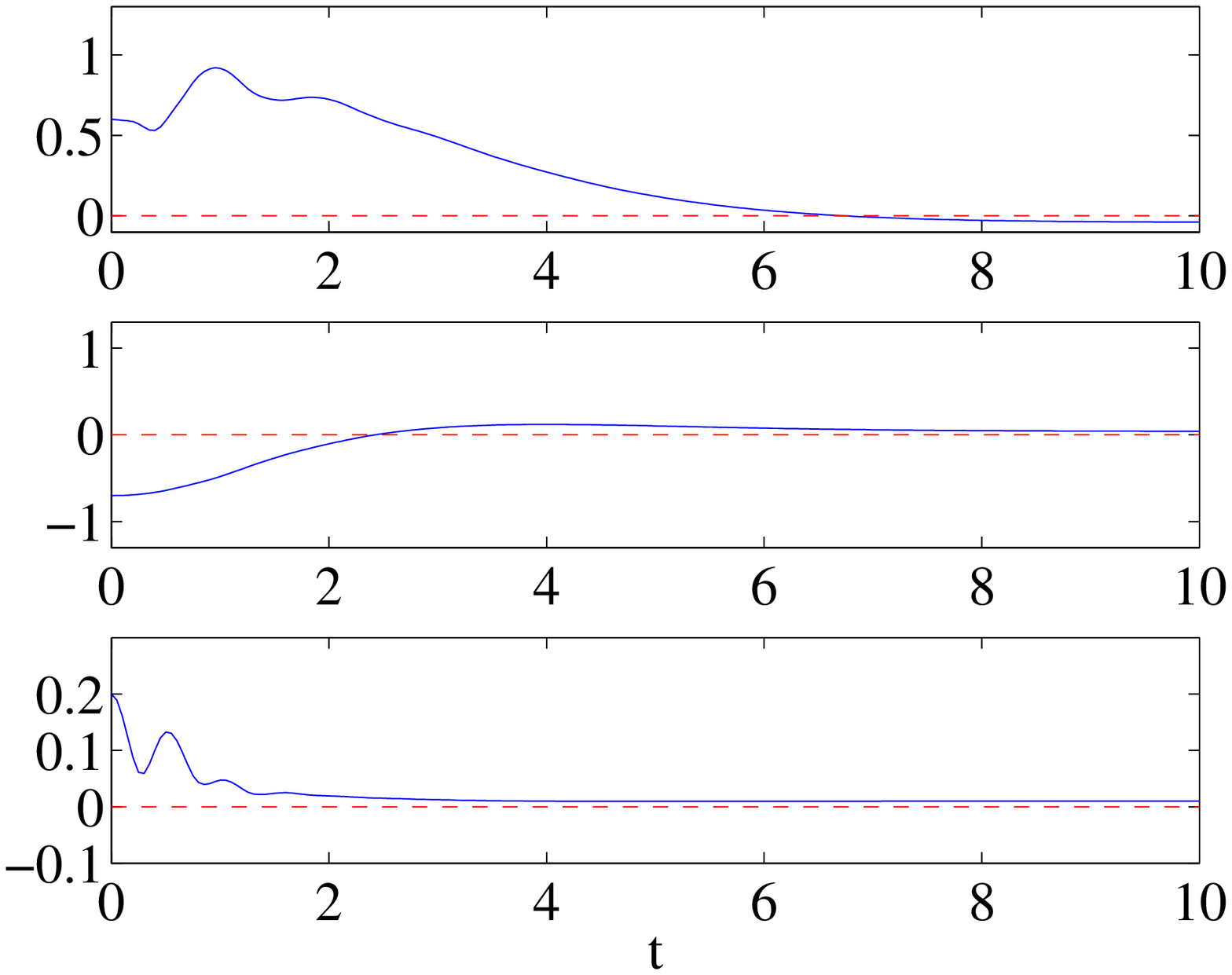}}\hspace{-0.5cm}
	\subfigure[Quadrotor velocity]{
		\includegraphics[width=0.6\columnwidth]{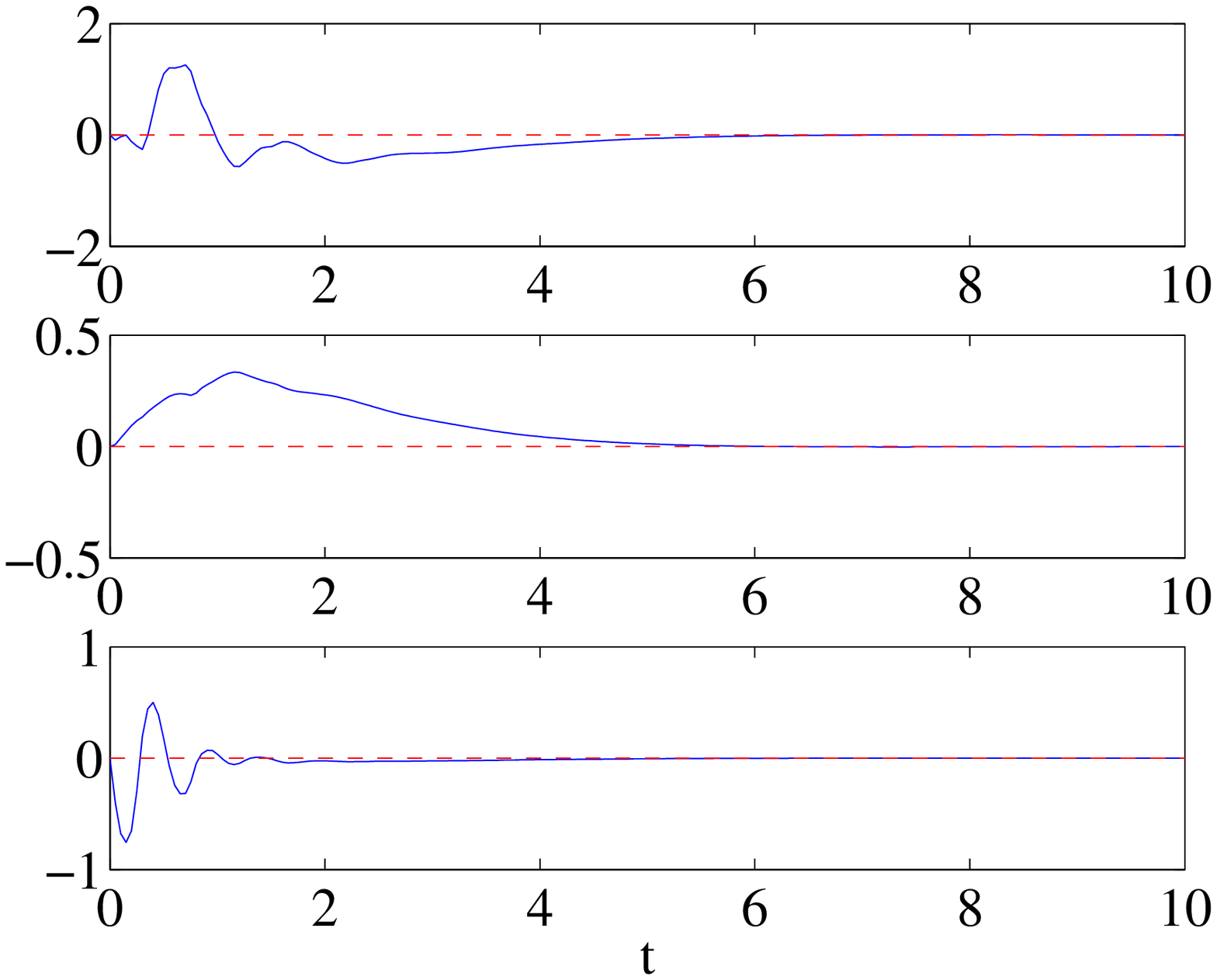}}
}
\caption{Stabilization of a payload connected to a quadrotor with 5 links with Integral term}\label{fig:simresultsW}
\end{figure}
Simulation results are illustrated at Figures \ref{fig:simresultsWO} and \ref{fig:simresultsW} where quad rotor stabilize the payload while reducing the direction error and the angular velocity error of the link. 
The corresponding maneuvers of the quadrotor and the links are illustrated by snapshots at Figure \ref{animationsim}. 
We considered two cases for this numerical simulation to compare the effect of the proposed integral term in the presence of disturbances as follows: (i) with integral term and (ii) without integral term, to emphasize the effect of the integral term. Comparison between Figure \ref{fig:simresultsWO} and \ref{fig:simresultsW} shows that the integral terms eliminates the steady state error significantly in presence of fixed disturbances where the position $x$ of the quadrotor converges to the desired value $x_d$ while stabilizing the payload and links below the quadrotor.

\begin{figure}
\centering
\subfigure[$ t =0 $]
{
\includegraphics[width=0.5in]{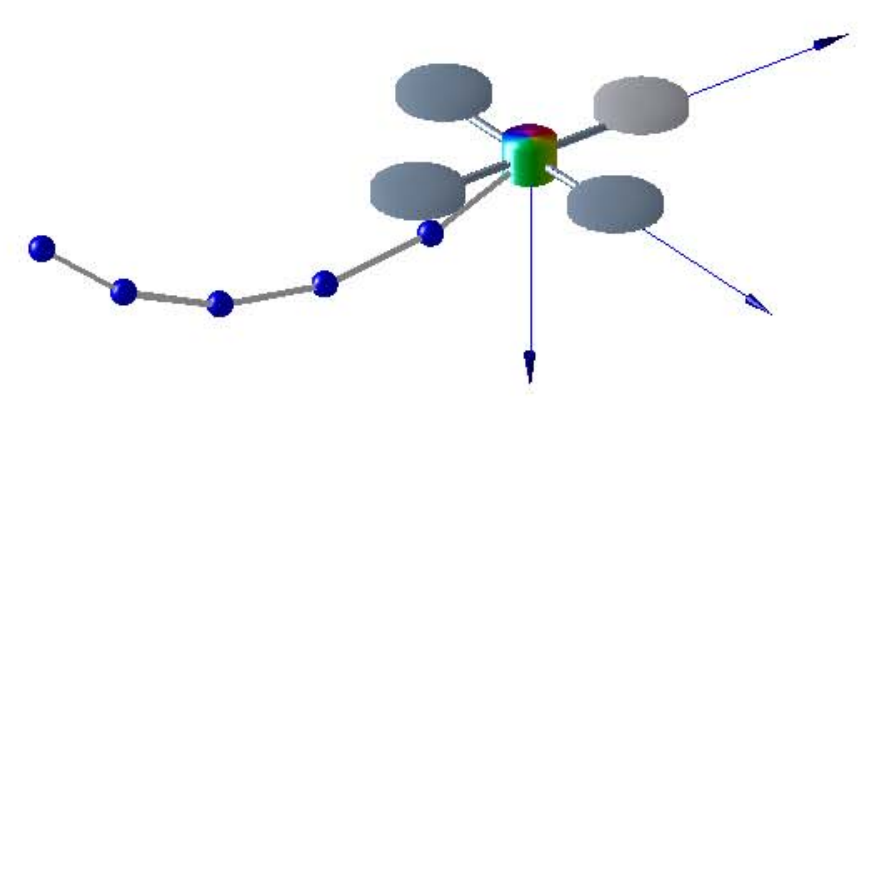}
\label{fig:fisrt_sub11}
}
\subfigure[$ t =0.2 $]
{
\includegraphics[width=0.5in]{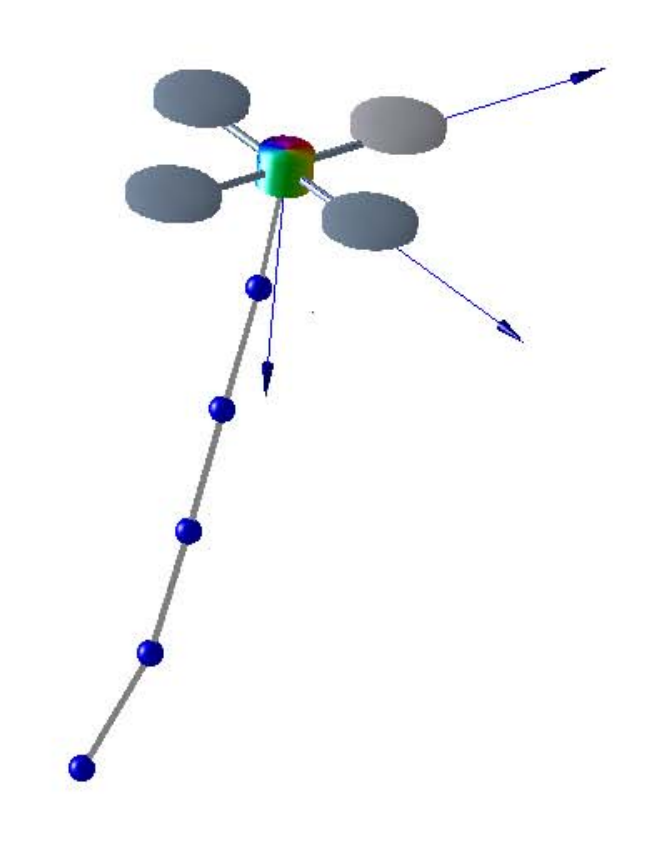}
}
\subfigure[$ t =0.35 $]
{
\includegraphics[width=0.5in]{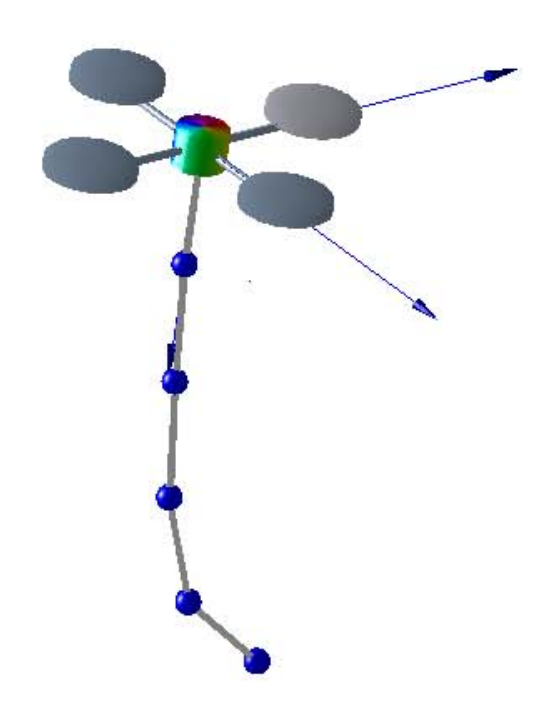}
}
\subfigure[$ t =0.40 $]
{
\includegraphics[width=0.5in]{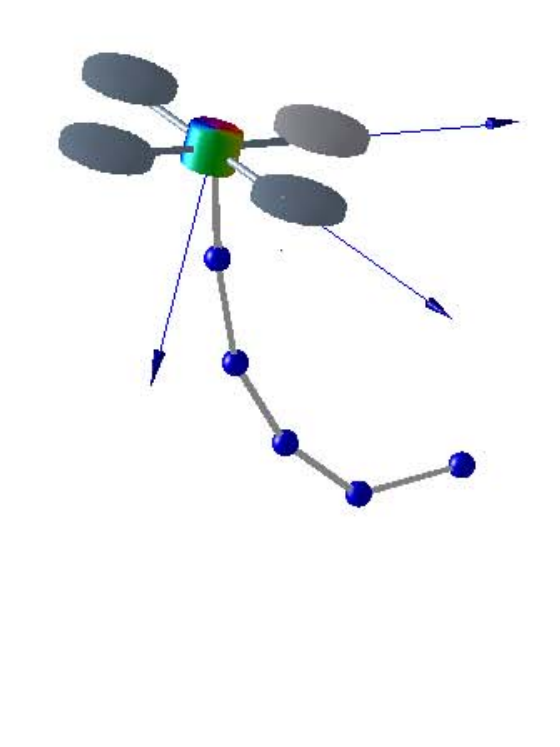}
}
\subfigure[$ t =0.42 $]
{
\includegraphics[width=0.5in]{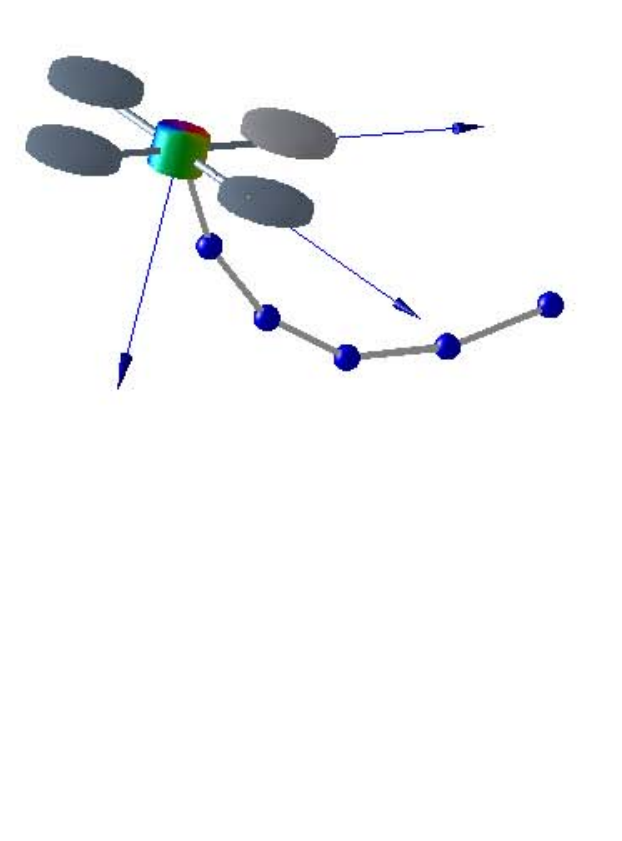}
}\\
\subfigure[$ t =0.45 $]
{
\includegraphics[width=0.5in]{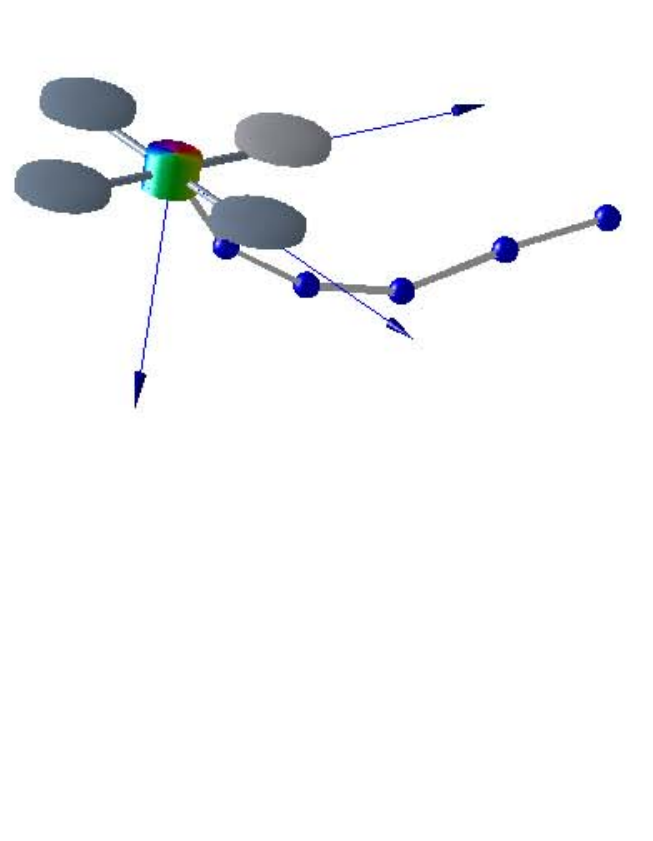}
}
\subfigure[$ t =0.5 $]
{
\includegraphics[width=0.5in]{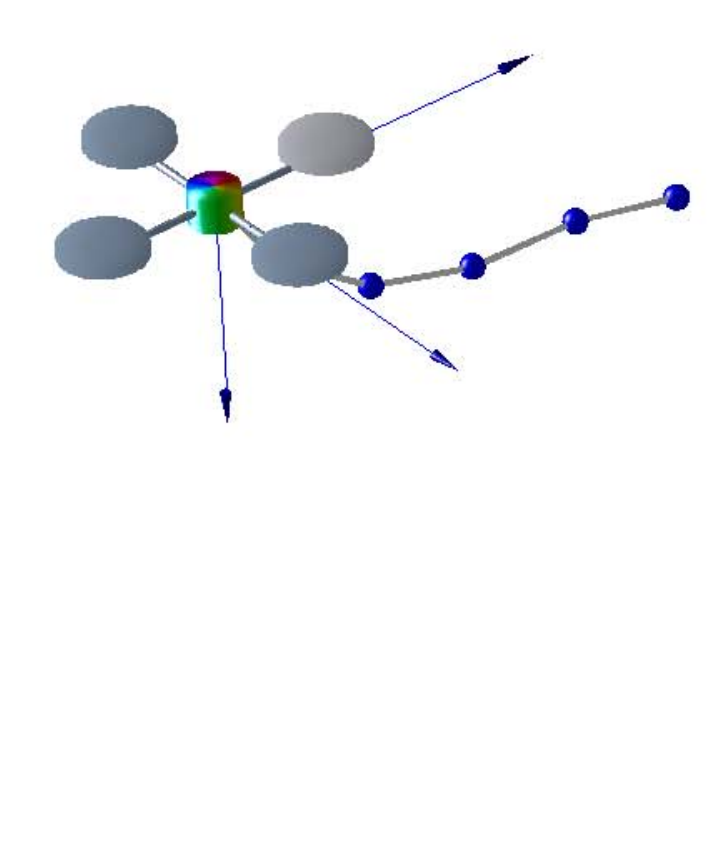}
}
\subfigure[$ t =0.6 $]
{
\includegraphics[width=0.5in]{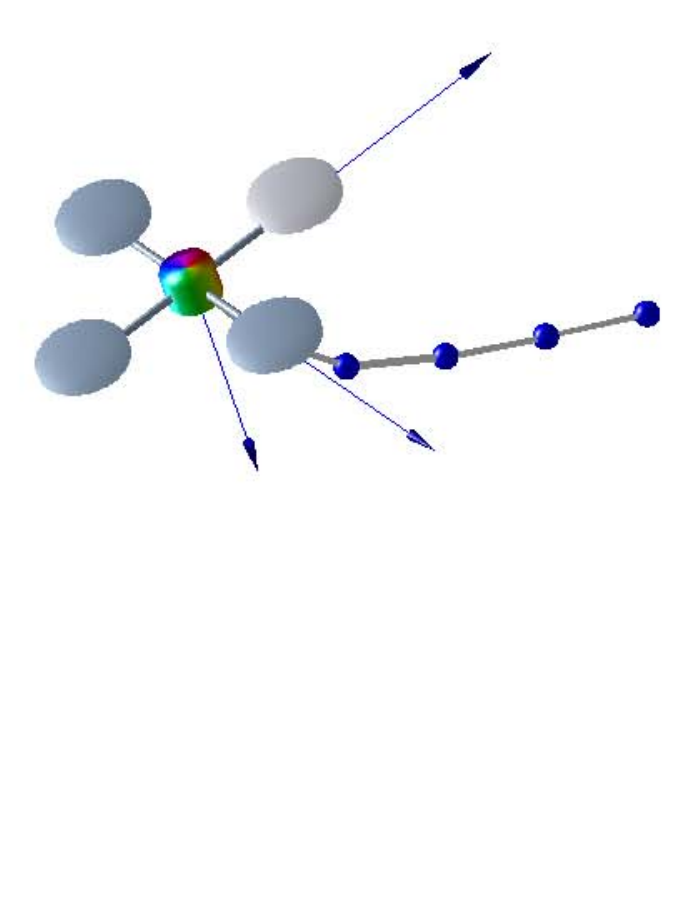}
}
\subfigure[$ t =0.7 $]
{
\includegraphics[width=0.5in]{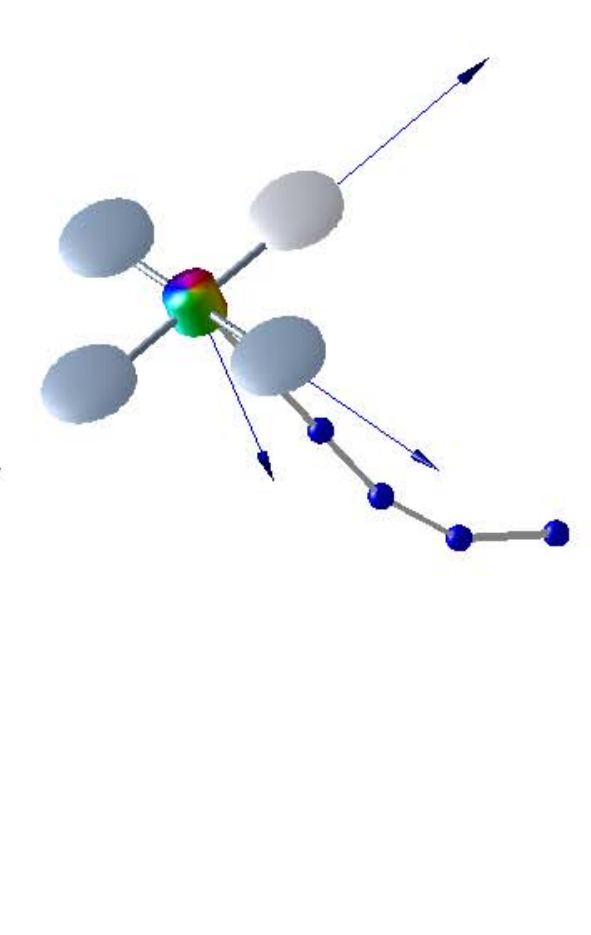}
}
\subfigure[$ t =0.8 $]
{
\includegraphics[width=0.5in]{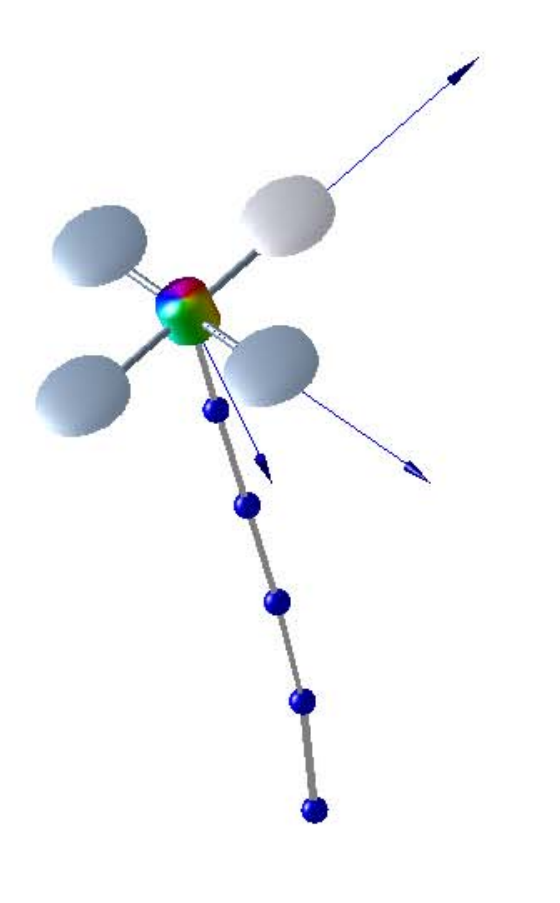}
}\\\subfigure[$ t =0.9 $]
{
\includegraphics[width=0.5in]{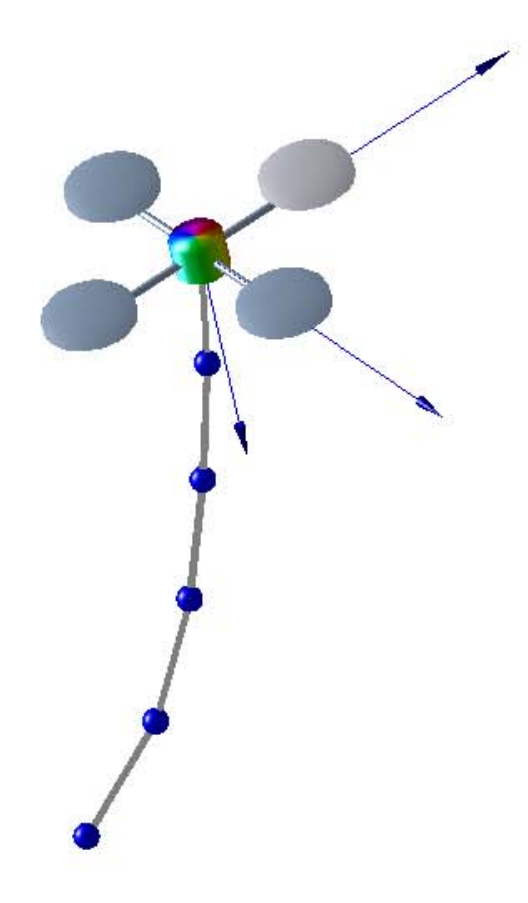}
}
\subfigure[$ t =1.3 $]
{
\includegraphics[width=0.5in]{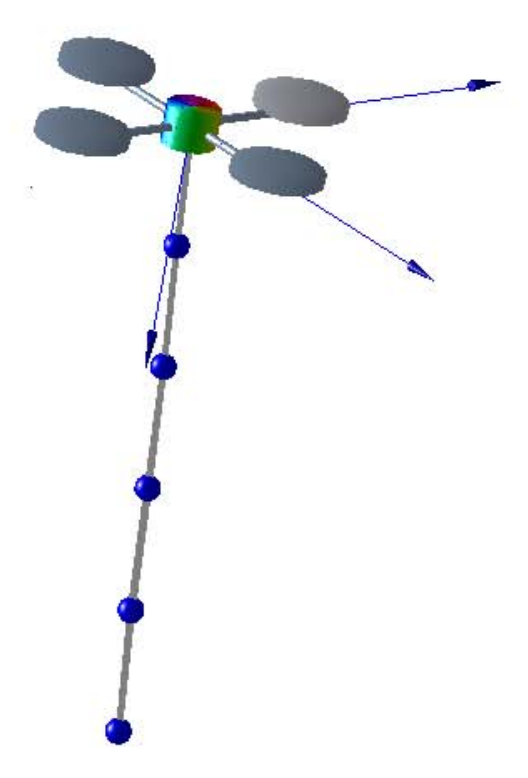}
}
\subfigure[$ t =1.5 $]
{
\includegraphics[width=0.5in]{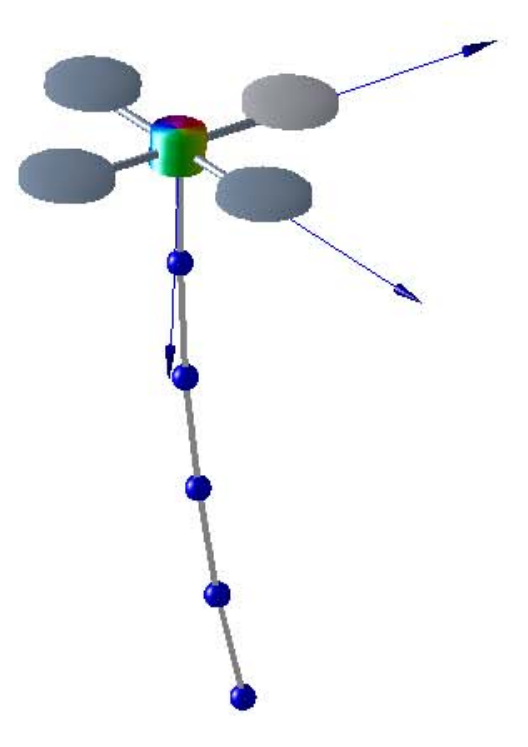}
}
\subfigure[$ t =2.0 $]
{
\includegraphics[width=0.5in]{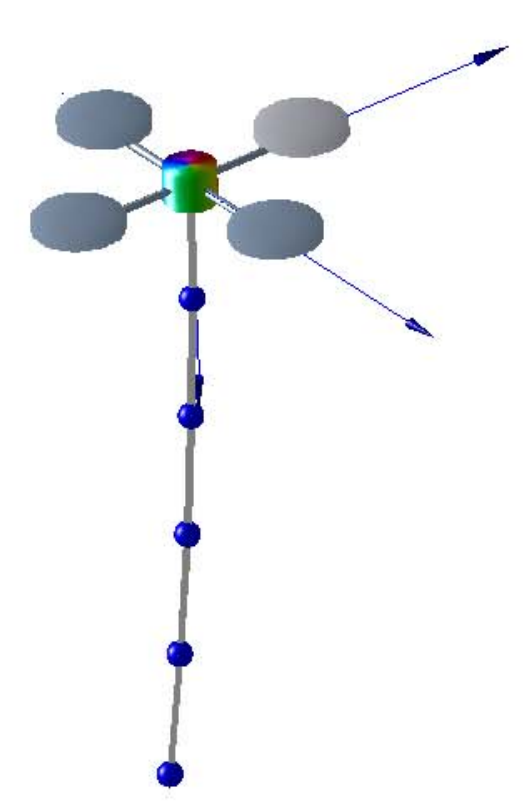}
}
\subfigure[$ t =10.0 $]
{
\includegraphics[width=0.5in]{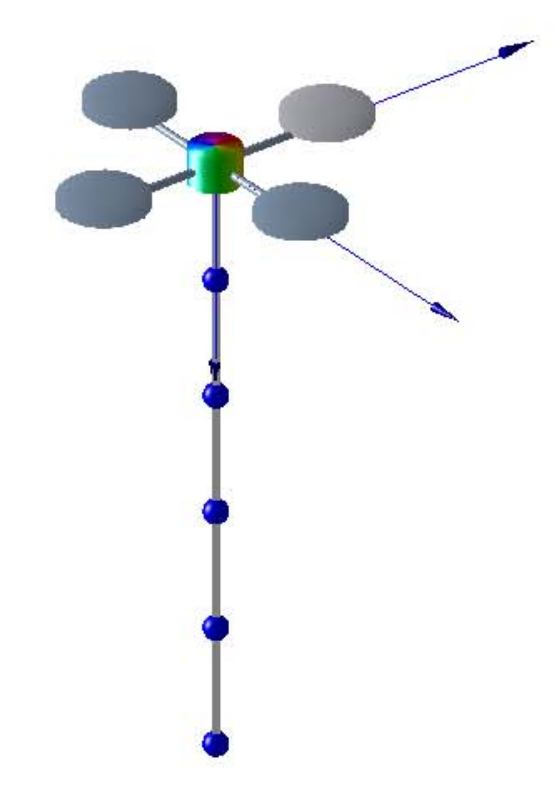}
}
\caption{Snapshots of the controlled maneuver}
\label{animationsim}
\end{figure}
\section{Experimental Results}\label{sec:ER}
Experimental results of the proposed controller are presented in this section. A quadrotor UAV is developed with the following configuration as illustrated at Figure \ref{fig:QuadHW}:
\begin{itemize}
\item Gumstix Overo computer-in-module (OMAP 600MHz processor), running a non-realtime Linux operating system. It is connected to a ground station via WIFI.
\item Microstrain 3DM-GX3 IMU, connected to Gumstix via UART.
\item BL-CTRL 2.0 motor speed controller, connected to Gumstix via I2C.
\item Roxxy 2827-35 Brushless DC motors.
\item XBee RF module, connected to Gumstix via UART.
\end{itemize}

The weight of the entire UAV system
is $0.791\mathrm{kg}$ including one battery.  A payload with mass of $m_1=0.036\ \mathrm{kg}$ is attached to the quadrotor via a cable of length $l_1=0.7\ \mathrm{m}$. The length from the center of the quadrotor to each motor rotational axis is $d=0.169\mathrm{m}$, the thrust to torque coefficient is $c_{{\tau}_f}=0.1056\mathrm{m}$ and the moment of inertia is $J=[0.56,0.56,1.05]\times 10^{-2}\,\mathrm{kgm^2}$. The angular velocity is measured from inertial measurement unit (IMU) and the attitude is estimated from IMU data. Position of the UAV is measured from motion capture system (Vicon) and the velocity is estimated from the measurement. Ground computing system receives the Vicon data and send it to the UAV via XBee. The Gumstix is adopted as micro computing unit on the UAV. It has two main threads, Vicon thread and IMU thread. The Vicon thread receives the Vicon measurement and estimates linear velocity of the quadrotor and runs at 30Hz. In IMU thread, it receives the IMU measurement and estimates the angular velocity. Also, control outputs are calculated at 120Hz in this thread.

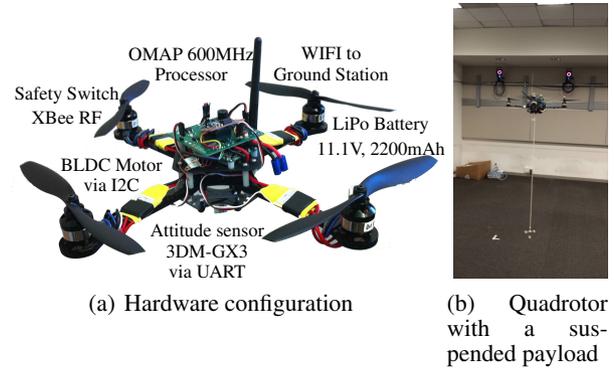
\begin{figure}
\centerline{
	\subfigure[Hardware configuration]{
\setlength{\unitlength}{0.1\columnwidth}\scriptsize
\begin{picture}(7,4)(0,0)
\put(0,0){\includegraphics[width=0.7\columnwidth]{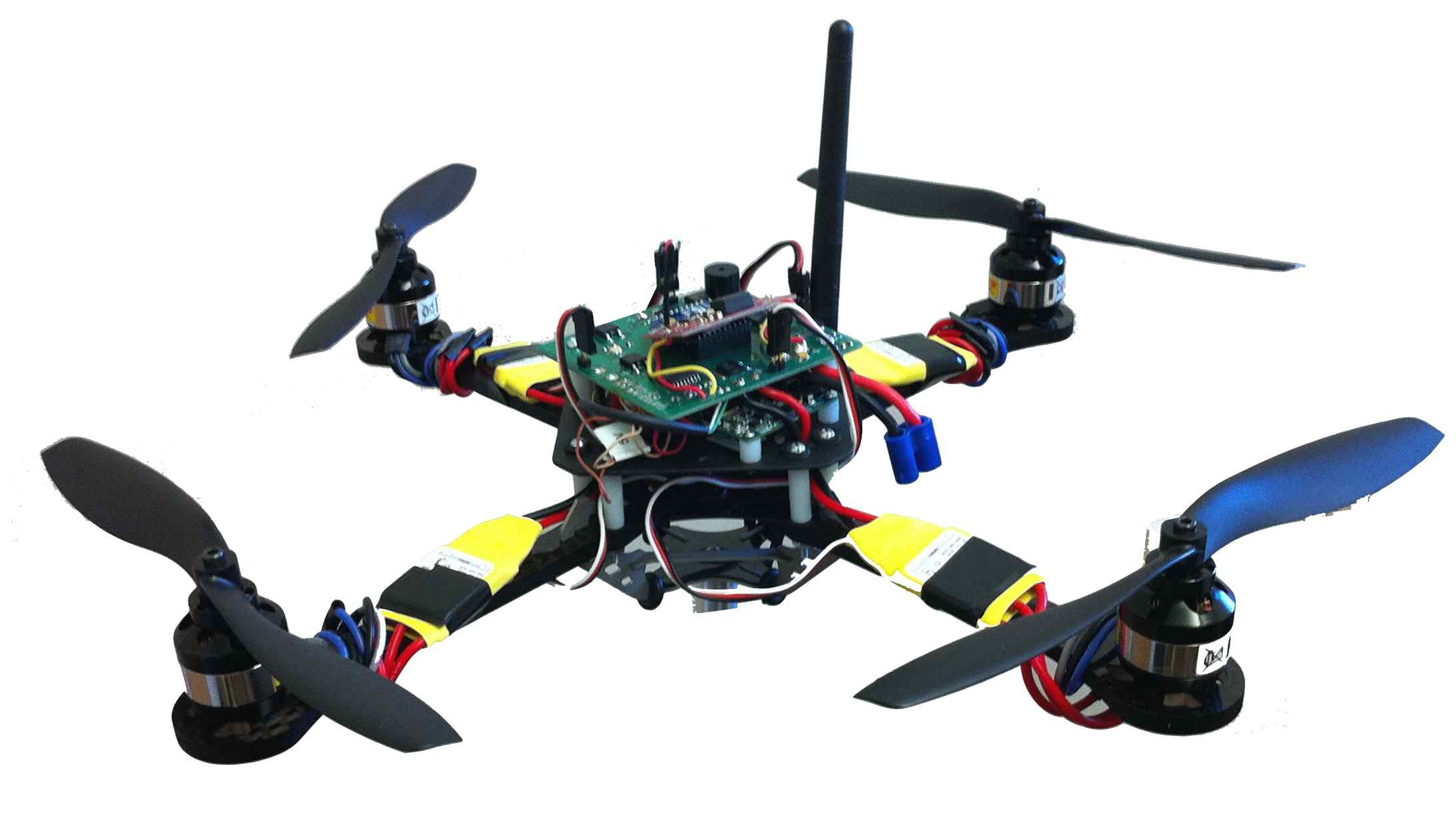}}
\put(1.95,3.2){\shortstack[c]{OMAP 600MHz\\Processor}}
\put(2.3,0){\shortstack[c]{Attitude sensor\\3DM-GX3\\ via UART}}
\put(0.85,1.4){\shortstack[c]{BLDC Motor\\ via I2C}}
\put(0.1,2.5){\shortstack[c]{Safety Switch\\XBee RF}}
\put(4.3,3.2){\shortstack[c]{WIFI to\\Ground Station}}
\put(5,2.0){\shortstack[c]{LiPo Battery\\11.1V, 2200mAh}}
\end{picture}}
	\subfigure[Quadrotor with a suspended payload]{
	\includegraphics[width=0.25\columnwidth]{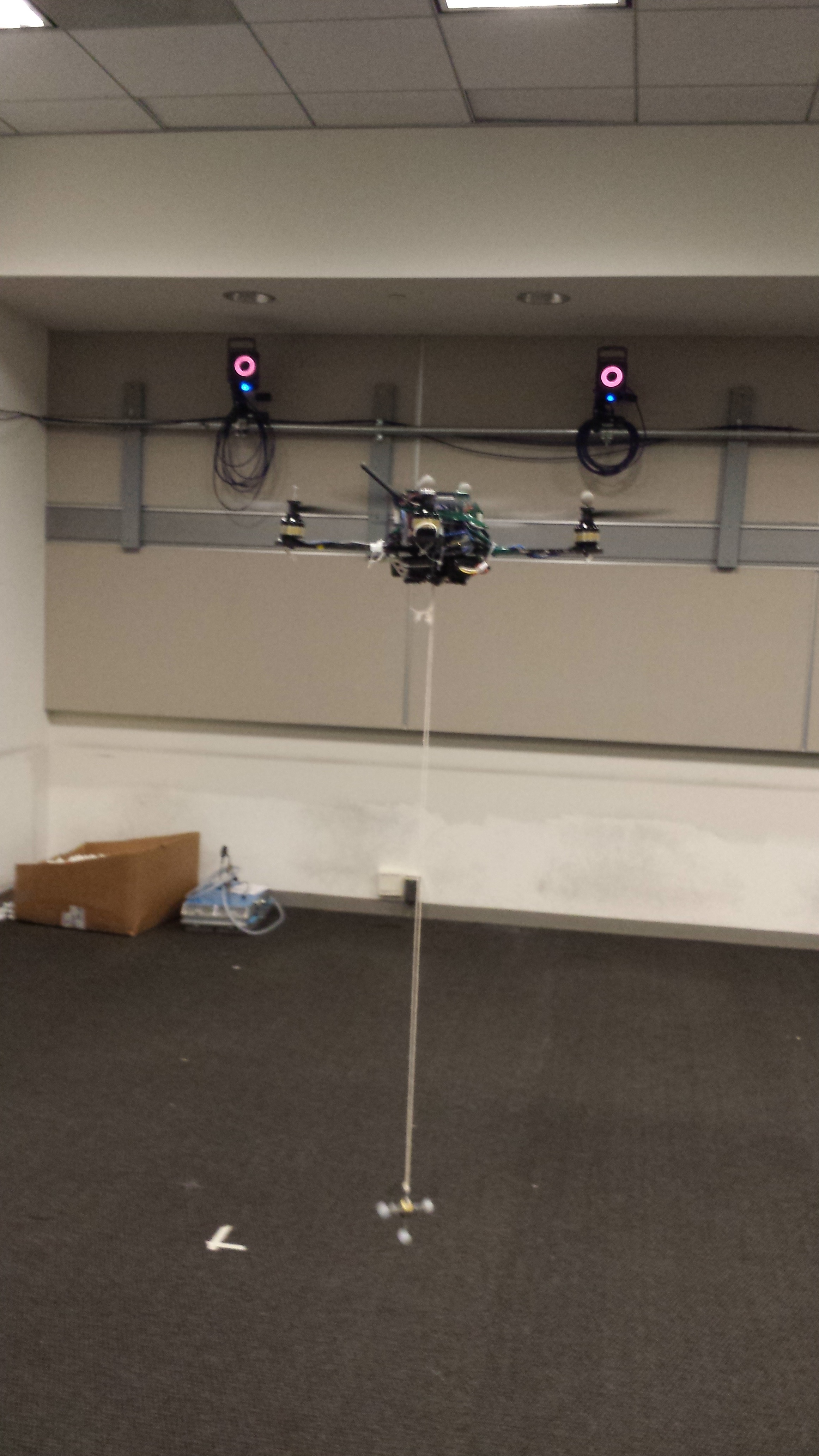}}
}
\caption{Hardware development for a quadrotor UAV}\label{fig:QuadHW}
\end{figure}

Two cases are considered and compared. For the first case, a position control system developed in~\cite{GooLeePECC13}, for quadrotor UAV that does not include the dynamics of the payload and the link, is applied to hover the quadrotor at the desired location, and the second case, the proposed control system is used.

\begin{figure}
\centerline{
\subfigure[Case I: quadrotor position control system~\cite{GooLeePECC13}]
{\includegraphics[width=1.2\columnwidth]{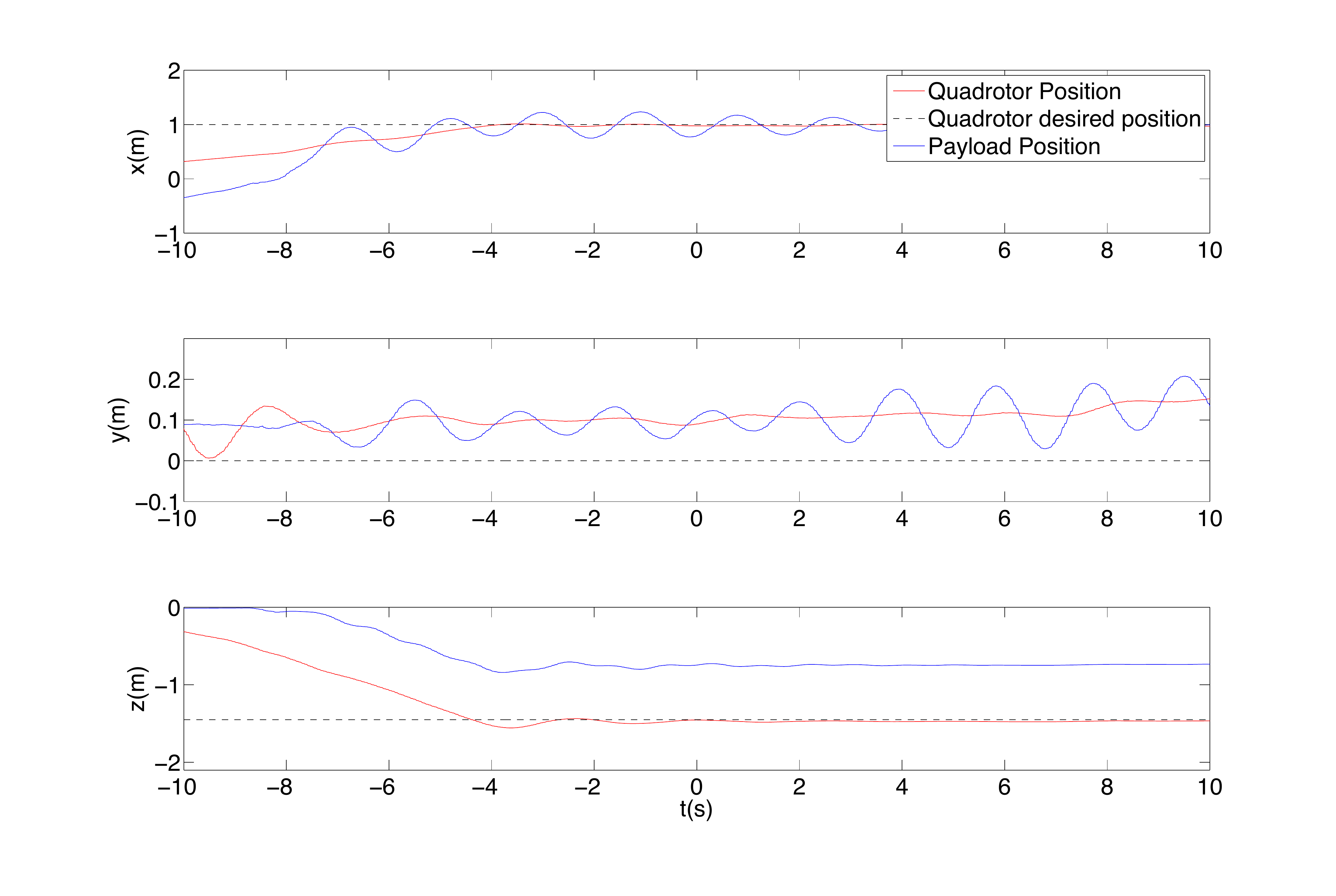}}}
\centerline{
\subfigure[Case II: proposed control system for quadrotor with suspended payload]
{\includegraphics[width=1.2\columnwidth]{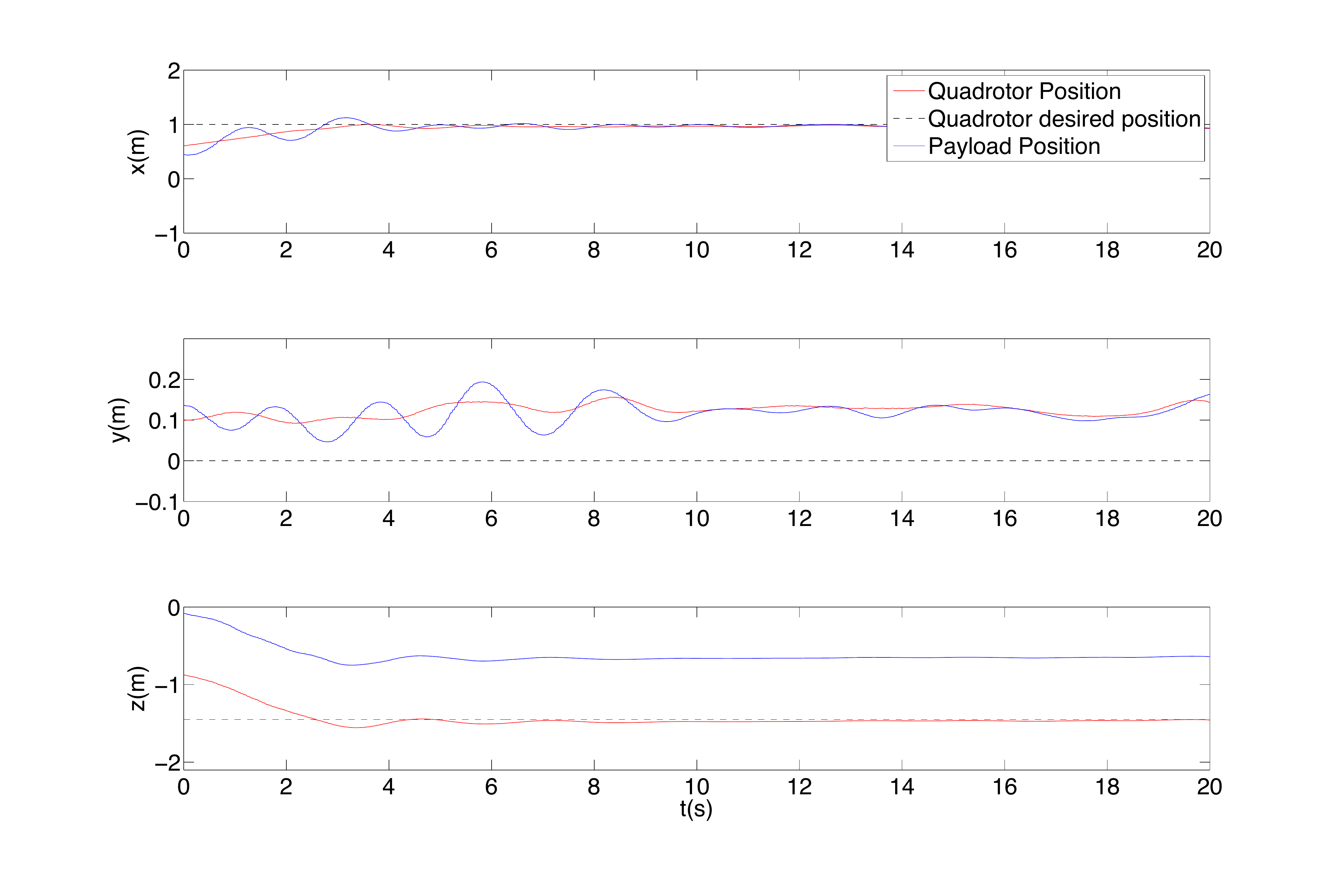}}}
\caption{Experimental results ($x_d$:black, $x$:red, $x+l_1q_1$:blue)}
\label{expresultsp}
\end{figure}


Experimental results are shown at Figures \ref{expresultsp} and \ref{expresultsq}. The position of the quadrotor and the payload is compared with the desired position of the quadrotor at Figure~\ref{expresultsp}, and the deflection angle of the link from the vertical direction are illustrated at Figure~\ref{expresultsq}. It is shown that the proposed control system reduces the undesired oscillation of the link effectively, compared with the quadrotor position control system.\footnote{A short video file of the experiments is also available at the experiment section of \url{http://fdcl.seas.gwu.edu}.}

\begin{figure}
\subfigure[Case I: quadrotor position control system~\cite{GooLeePECC13}]
{\includegraphics[width=0.49\columnwidth]{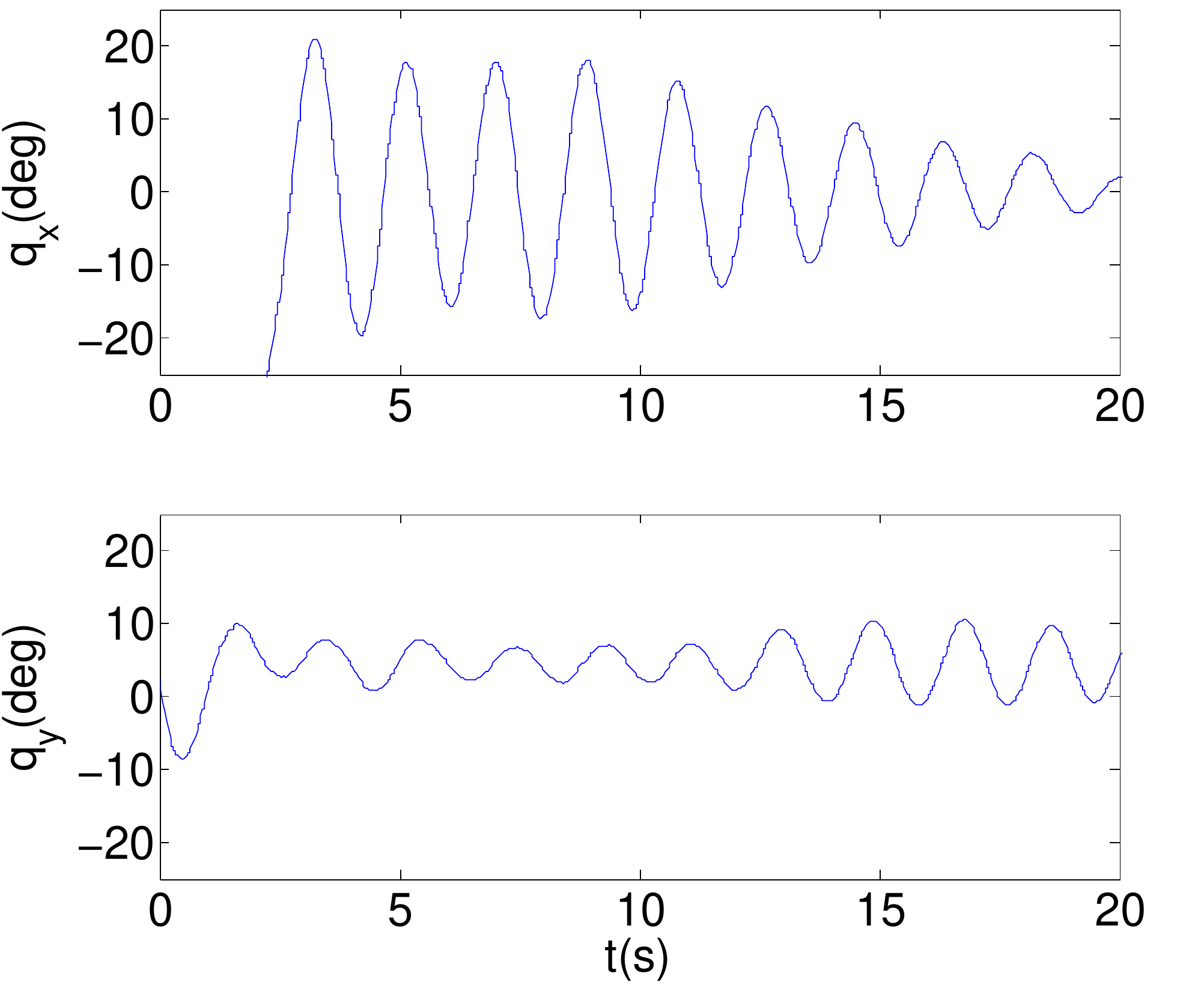}}
\subfigure[Case II: proposed control system for quadrotor with suspended payload]
{\includegraphics[width=0.49\columnwidth]{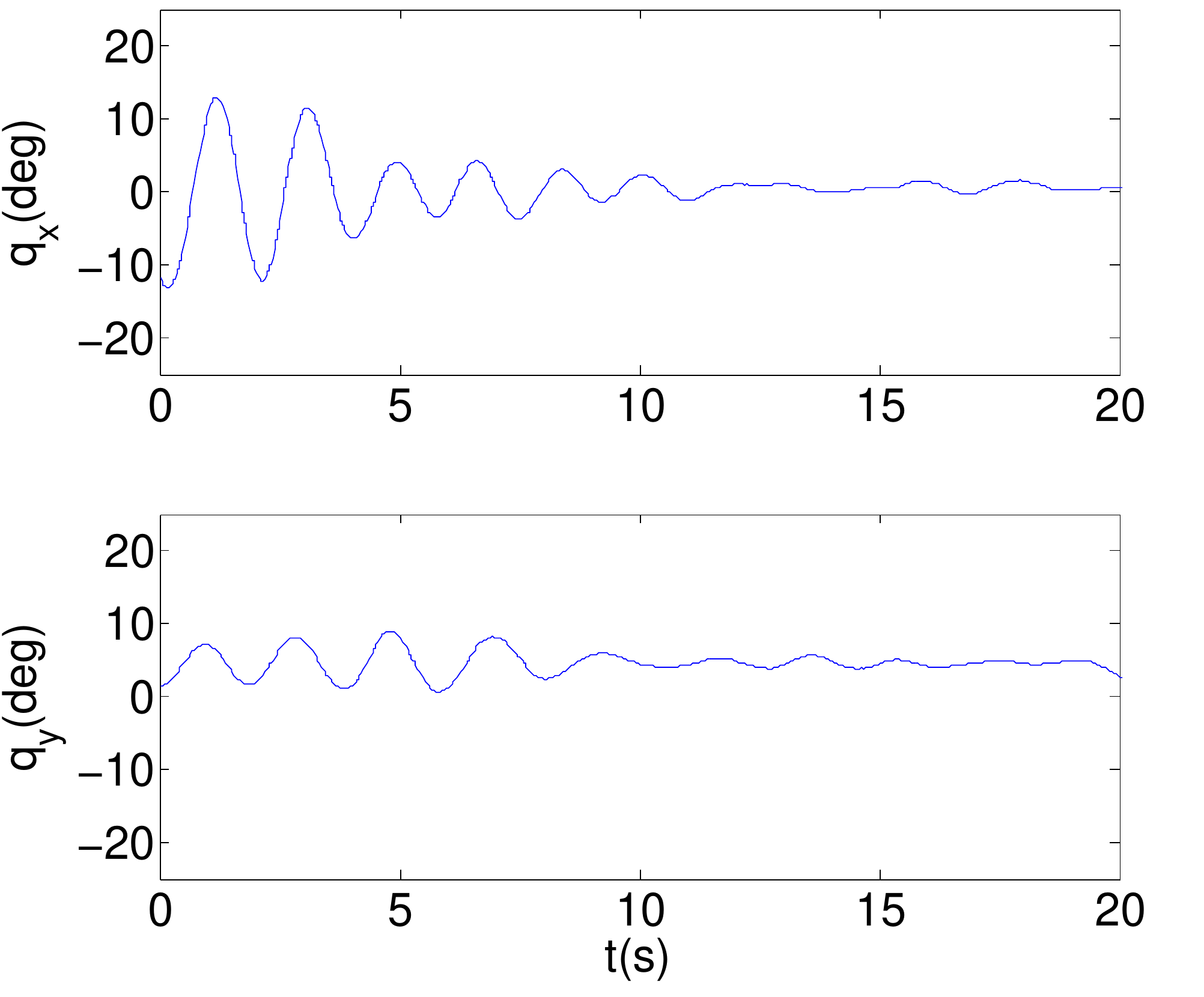}}
\caption{Experimental results: link deflection angles}
\label{expresultsq}
\end{figure}

\section{Conclusions}

Euler-Lagrange equations have been used for the quadrotor and the chain pendulum to model a flexible cable transporting a load in 3D space. These derivations developed in a remarkably compact form which allows us to choose arbitrary number and any configuration of the links. We developed a geometric nonlinear controller to stabilize the links below the quadrotor in the equilibrium position from any chosen initial condition. We expanded these derivations in such way that there is no need of using local angle coordinate and this advantageous technique signalize our derivations.

\appendix
\subsection{Proof for Proposition \ref{prop:FDM}}\label{sec:PfFDM}
From \refeqn{KE} and \refeqn{PE}, the Lagrangian is given by
\begin{align}
L&=\frac{1}{2}M_{00}\|\dot{x}\|^{2}+\dot{x}\cdot\sum^{n}_{i=1}{M_{0i}\dot{q}_i}+\frac{1}{2}\sum^{n}_{i,j=1}{M_{ij}\dot{q}_{i}\cdot\dot{q}_{j}}\nonumber \\
&\quad+\sum^{n}_{i=1}\sum^{n}_{a=i}m_{a}gl_{i}e_{3}\cdot q_{i}+M_{00}ge_{3}\cdot x+\frac{1}{2}\Omega^{T}J\Omega,
\end{align}
The derivatives of the Lagrangian are given by
\begin{align*}
\D_x L & = M_{00} g e_3,\\
\D_{\dot x} L & = M_{00} \dot x + \sum_{i=1}^n M_{0i}\dot q_i,
\end{align*}
where $\D_x L$ represents the derivative of $L$ with respect to $x$. From the variation of the angular velocity given after \refeqn{delR}, we have
\begin{align}
\D_{\Omega}L\cdot\delta\Omega=J\Omega\cdot(\dot\eta+\Omega\times\eta)
=J\Omega\cdot\dot\eta- \eta\cdot(\Omega\times J\Omega).
\end{align}
Similarly from \refeqn{delqi}, the derivative of the Lagrangian with respect to $q_i$ is given by
\begin{align*}
\D_{q_i} L \cdot \delta q_i = \sum_{a=i}^n m_a gl_ie_3\cdot (\xi_i\times q_i) 
= -\sum_{a=i}^n m_a gl_i\hat e_3 q_i\cdot \xi_i.
\end{align*}
The variation of $\dot q_i$ is given by
\begin{align*}
\delta\dot q_i = \dot \xi_i\times q_i + \xi_i\times q_i.
\end{align*}
Using this, the derivative of the Lagrangian with respect to $\dot q_i$ is given by
\begin{align*}
&\D_{\dot q_i}L\cdot \delta \dot q_i  = (M_{i0}\dot x + \sum_{j=1}^n M_{ij}\dot q_j) \cdot \delta \dot q_i \\
& = (M_{i0}\dot x + \sum_{j=1}^n M_{ij}\dot q_j) \cdot (\dot \xi_i \times q + \xi_i \times \dot q_i)\\
& = \hat q_i (M_{i0}\dot x + \sum_{j=1}^n M_{ij}\dot q_j)\cdot \dot\xi_i +
\hat{\dot q}_i (M_{i0}\dot x + \sum_{j=1}^n M_{ij}\dot q_j)\cdot \xi_i.
\end{align*}
Let $\mathfrak{G}$ be the action integral, i.e., $\mathfrak{G}=\int_{t_0}^{t_f} L\,dt$. From the above expressions for the derivatives of the Lagrangian, the variation of the action integral can be written as
\begin{align*}
\delta \mathfrak{G} =& \int_{t_0}^{t_f} \{M_{00} \dot x + \sum_{i=1}^n M_{0i}\dot q_i\}\cdot \delta \dot x
+M_{00}ge_3\cdot\delta x\\
& +\sum_{i=1}^n \{\hat q_i (M_{i0}\dot x + \sum_{j=1}^n M_{ij}\dot q_j)\}\cdot \dot \xi_i\\
 &+ \sum_{i=1}^n\{\hat{\dot q}_i (M_{i0}\dot x + \sum_{j=1}^n M_{ij}\dot q_j)-\sum_{a=i}^n m_a gl_i\hat e_3 q_i \}\cdot \xi_i\\
 &+ J\Omega\cdot\dot\eta- \eta\cdot(\Omega\times J\Omega)\,dt.
\end{align*}
Integrating by parts and using the fact that variations at the end points vanish, this reduces to
\begin{align*}
\delta \mathfrak{G} =& \int_{t_0}^{t_f} \{M_{00}ge_3 -M_{00} \ddot x - \sum_{i=1}^n M_{0i}\ddot q_i\}\cdot \delta x \\
&+ \sum_{i=1}^n\{-\hat q_i (M_{i0}\ddot x + \sum_{j=1}^n M_{ij}\ddot q_j)-\sum_{a=i}^n m_a gl_i\hat e_3 q_i \}\cdot \xi_i\\
& - \eta\cdot(J\dot\Omega+\Omega\times J\Omega) \,dt.
\end{align*}
According to the Lagrange-d'Alembert principle, the variation of the action integral is equal to the negative of the virtual work done by the external force and moment, namely
\begin{align}\label{eqn:estef}
-\int_{t_0}^{t_f} (-fRe_3+\Delta_{x})\cdot\delta x + (M+\Delta_R)\cdot \eta\; dt,
\end{align}
and we obtain \refeqn{xddot} and \refeqn{Wdot}. As $\xi_i$ is perpendicular to $q_i$, we also have
\begin{gather}
-\hat q_i^2 (M_{i0}\ddot x + \sum_{j=1}^n M_{ij}\ddot q_j)+\sum_{a=i}^n m_a gl_i\hat q_i^2 e_3=0.\label{eqn:qddot0}
\end{gather}
Equation \refeqn{qddot0} is rewritten to obtain an explicit expression for $\ddot q_i$. As $q_i\cdot \dot q_i =0$, we have $\dot q_i \cdot \dot q_i +q_i\cdot \ddot q_i=0$. Using this, we have
\begin{align*}
-\hat q_i^2 \ddot q_i = -(q_i\cdot \ddot q_i )q_i + (q_i\cdot q_i)\ddot q_i =(\dot q_i \cdot \dot q_i) q_i + \ddot q_i.
\end{align*}
Substituting this equation into \refeqn{qddot0}, we obtain \refeqn{qddot}. This can be slightly rewritten in terms of the angular velocities. Since $\dot q_i = \omega_i\times q_i$ for the angular velocity $\omega_i$ satisfying $q_i\cdot\omega_i=0$, we have
\begin{align*}
    \ddot q_i & = \dot \omega_i \times q_i + \omega_i\times (\omega_i\times q_i) \\
    &= \dot \omega_i \times q_i - \|\omega_i\|^2q_i=-\hat q_i \dot\omega_i - \|\omega_i\|^2q_i.
\end{align*}
Using this and the fact that $\dot\omega_i\cdot q_i=0$, we obtain \refeqn{ELwm}.

\subsection{Proof for Proposition \ref{prop:FDMM}}\label{sec:PfFDMM}

The variations of $x,u$ and $q$ are given by \refeqn{delxLin} and \refeqn{delqLin}. From the kinematics equation $\dot q_i=\omega_i\times q_i$, $\delta\dot q_i$ is given by
\begin{align*}
\delta \dot q_i = \dot\xi_i \times e_3 =\delta\omega_i \times e_3 + 0\times (\xi_i\times e_3)=\delta\omega_i \times e_3.
\end{align*}
Since both sides of the above equation is perpendicular to $e_3$, this is equivalent to $e_3\times(\dot\xi_i\times e_3) = e_3\times(\delta\omega_i\times e_3)$, which yields
\begin{gather*}
\dot \xi - (e_3\cdot\dot\xi) e_3 = \delta\omega_i -(e_3\cdot\delta\omega_i)e_3.
\end{gather*}
Since $\xi_i\cdot e_3 =0$, we have $\dot\xi\cdot e_3=0$. As $e_3\cdot\delta\omega_i=0$ from the constraint, we obtain the linearized equation for the kinematics equation:
\begin{align}
\dot\xi_i = \delta\omega_i.\label{eqn:dotxii}
\end{align}
Substituting these into \refeqn{ELwm}, and ignoring the higher order terms, we obtain \refeqn{Lin}. See Appendix 4 for details.
\subsection{Proof for Proposition \ref{prop:SA1}}\label{sec:stability1}
We first show stability of the rotational dynamics, and later it is combined with the stability analysis of the translational dynamics of quad rotor and the rotational dynamics of links. 
\subsubsection{a) Attitude Error Dynamics}
Here, attitude error dynamics for $e_{R}$, $e_{\Omega}$ are derived and we find conditions on control parameters to guarantee the boundedness of attitude tracking errors. The time-derivative of $Je_{\Omega}$ can be written as
\begin{align}
J\dot e_\Omega & = \{Je_\Omega + d\}^\wedge e_\Omega - k_R e_R-k_\Omega e_\Omega- k_I e_I + \Delta_R,\label{eqn:JeWdot}
\end{align}
where $d=(2J-\trs{J}I)R^TR_d\Omega_d\in\Re^3$~\cite{TFJCHTLeeHG}. The important property is that the first term of the right hand side is normal to $e_\Omega$, and it simplifies the subsequent Lyapunov analysis.

\subsubsection{b) Stability for Attutide Dynamics}
Define a configuration error function on $\SO$ as follows
\begin{align}
\Psi= \frac{1}{2}\trs[I- R_c^T R].
\end{align}
We introduce the following Lyapunov function
\begin{align}
\begin{split}
\mathcal{V}_{2}=&\frac{1}{2}e_{\Omega}\cdot J\dot{e}_{\Omega}+k_{R}\Psi(R,R_{d})+c_{2}e_{R}\cdot e_{\Omega}\\
&+\frac{1}{2}k_{I}\|e_{I}-\frac{\Delta_R}{k_{I}}\|^{2}.
\end{split}
\end{align}
Consider a domain $D_{2}$ given by
\begin{align}
D_2 = \{ (R,\Omega)\in \SO\times\Re^3\,|\, \Psi(R,R_d)<\psi_2<2\}.\label{eqn:D2}
\end{align}
In this domain we can show that $\mathcal{V}_{2}$ is bounded as follows~\cite{TFJCHTLeeHG}
\begin{align}\label{eqn:ffff1}
\begin{split}
z_{2}^{T}M_{21}z_{2}&+\frac{k_{I}}{2}\|e_{I}-\frac{\Delta_R}{k_{I}}\|^{2}\leq\mathcal{V}_{2}\\
&\leq z_{2}^{T}M_{22}z_{2}+\frac{k_{I}}{2}\|e_{I}-\frac{\Delta_R}{k_{I}}\|^{2},
\end{split}
\end{align}
where $z_{2}=[\|e_{R}\|,\|e_{\Omega}\|]^{T}\in \Re^{2}$ and the matrices $M_{21}$, $M_{22}$ are given by
\begin{align}
M_{21}=\frac{1}{2}\begin{bmatrix}
k_{R}&-c_{2}\lambda_{M}\\
-c_{2}\lambda_{M}&\lambda_{m}
\end{bmatrix},M_{22}=\frac{1}{2}\begin{bmatrix}
\frac{2k_{R}}{2-\psi_{2}}&c_{2}\lambda_{M}\\
c_{2}\lambda_{M}&\lambda_{M}
\end{bmatrix},
\end{align}
The time derivative of $\mathcal{V}_2$ along the solution of the controlled system is given by
\begin{align*}
\dot{\mathcal{V}}_2  =&
-k_\Omega\|e_\Omega\|^2 - e_\Omega\cdot(k_Ie_I-\Delta_R) \\
&+ c_2 \dot e_R \cdot Je_\Omega+ c_2 e_R \cdot J\dot e_\Omega + (k_Ie_I- \Delta_R) \dot e_I.
\end{align*}
We have $\dot e_I = c_2 e_R + e_\Omega$ from \refeqn{integralterm}. Substituting this and \refeqn{JeWdot}, the above equation becomes
\begin{align*}
\dot{\mathcal{V}}_2 =&
-k_\Omega\|e_\Omega\|^2  + c_2 \dot e_R \cdot Je_\Omega-c _2 k_R \|e_R\|^2 \\
&+ c_2 e_R \cdot ((Je_\Omega+d)^\wedge e_\Omega -k_\Omega e_\Omega).
\end{align*}
We have $\|e_R\|\leq 1$, $\|\dot e_R\|\leq \|e_\Omega\|$~\cite{TFJCHTLeeHG}, and choose a constant $B_{2}$ such that $\|d\|\leq B_2$. Then we have
\begin{align}
\dot{\mathcal{V}}_2 \leq - z_2^T W_2 z_2,\label{eqn:dotV2}
\end{align}
where the matrix $W_2\in\Re^{2\times 2}$ is given by
\begin{align*}
W_2 = \begin{bmatrix} c_2k_R & -\frac{c_2}{2}(k_\Omega+B_2) \\ 
-\frac{c_2}{2}(k_\Omega+B_2) & k_\Omega-2c_2\lambda_M \end{bmatrix}.
\end{align*}
The matrix $W_{2}$ is a positive definite matrix if 
\begin{align}\label{eqn:c2}
c_{2}<\min\{\frac{\sqrt{k_{R}\lambda_{m}}}{\lambda_{M}},\frac{4k_{\Omega}}{8k_{R}\lambda_{M}+(k_{\Omega}+B_{2})^{2}} \}.
\end{align}
This implies that
\begin{align}\label{eqn:eq2}
\dot{\mathcal{V}}_{2}\leq - \lambda_{m}(W_{2})\|z_{2}\|^{2},
\end{align} 
which shows stability of attitude dynamics.
\subsubsection{c) Translational Error Dynamics}
We derive the tracking error dynamics and a Lyapunov function for the translational dynamics of a quadrotor UAV and the dynamics of links. Later it is combined with the stability analyses of the rotational dynamics. This proof is based on the Lyapunov method presented in Theorem 3.6 and 3.7~\cite{Kha96}. From \refeqn{delxLin}, \refeqn{xddot}, \refeqn{Lin}, and \refeqn{fi}, the linearized equation of motion for the controlled full dynamic model is given by
\begin{align}\label{eqn:salam}
\Mb\ddot \xb  + \Gb\xb=\Bb(-fRe_{3}-M_{00}ge_{3})+\g(\xb,\dot{\xb})+\Bb\Delta_{x},
\end{align}
and $\g(\xb,\dot{\xb})$ is higher order term. The subsequent analyses are developed in the domain $D_{1}$
\begin{align}
D_1=\{&(\xb,\dot{\xb},R,e_\Omega)\in\Re^{2n+3}\times\Re^{2n+3}\times \SO\times\Re^3\,|\,\nonumber\\
& \Psi< \psi_1 < 1\}.\label{eqn:D}
\end{align}
In the domain $D_{1}$, we can show that 
\begin{align}
\frac{1}{2} \norm{e_R}^2 \leq  \Psi(R,R_c) \leq \frac{1}{2-\psi_1} \norm{e_R}^2\label{eqn:eRPsi1}.
\end{align}
Consider the quantity $e_{3}^{T}R_{c}^{T}Re_{3}$, which represents the cosine of the angle between $b_{3}=Re_{3}$ and $b_{3_{c}}=R_{c}e_{3}$. Since $1-\Psi(R,R_{c})$ represents the cosine of the eigen-axis rotation angle between $R_{c}$ and $R$, we have $e_{3}^{T}R_{c}^{T}Re_{3}\geq 1-\Psi(R,R_{c})>0$ in $D_{1}$. Therefore, the quantity $\frac{1}{e_{3}^{T}R_{c}^{T}Re_{3}}$ is well-defined. We add and subtract $\frac{f}{e_{3}^{T}R_{c}^{T}Re_{3}}R_{c}e_{3}$ to the right hand side of \refeqn{salam} to obtain
\begin{align}\label{eqn:taghall}
\Mb\ddot \xb  + \Gb\xb=&\Bb(\frac{-f}{e_{3}^{T}R_{c}^{T}Re_{3}}R_{c}e_{3}-X-M_{00}ge_{3}+\Delta_{x})+\g(\xb,\dot{\xb}),
\end{align}
where $X\in \Re^{3}$ is defined by
\begin{align}\label{eqn:Xdef}
X=\frac{f}{e_{3}^{T}R_{c}^{T}Re_{3}}((e_{3}^{T}R_{c}^{T}Re_{3})Re_{3}-R_{c}e_{3}).
\end{align}
The first term on the right hand side of \refeqn{taghall} can be written as 
\begin{align}
-\frac{f}{e_{3}^{T}R_{c}^{T}Re_{3}}R_{c}e_{3}=-\frac{(\|A\|R_{c}e_{3})\cdot Re_{3}}{e_{3}^{T}R_{c}^{T}Re_{3}}\cdot -\frac{A}{\|A\|}=A.
\end{align}
Substituting this and \refeqn{A} into \refeqn{taghall}
\begin{align}
\Mb\ddot \xb  + \Gb\xb=\Bb(-K_{x}\xb-K_{\dot{x}}\dot{\xb}-K_{z}\sat_{\sigma}(e_{\xb})-X+\Delta_{x})+\g(\xb,\dot{\xb}),
\end{align}
This can be rearranged as
\begin{align}
\begin{split}
\ddot{\xb}=&-(\Mb^{-1}\Gb+\Mb^{-1}\Bb K_{x})\xb-(\Mb^{-1}\Bb K_{\dot{x}})\dot{\xb}\\
&-\Mb^{-1}\Bb X-\Mb^{-1}\Bb K_{z}\sat_{\sigma}(e_{\xb})+\Mb^{-1}\g(\xb,\dot{\xb})+\Mb^{-1}\Bb\Delta_{x}.
\end{split}
\end{align}
Using the definitions for $\mathds{A}$, $\mathds{B}$, and $z_{1}$ presented before, the above expression can be rearranged as
\begin{align}\label{eqn:zdot1}
\dot{z}_{1}=&\mathds{A} z_{1}+\mathds{B}(-\Bb X+\g(\xb,\dot{\xb})-\Bb K_{z}\sat_{\sigma}(e_{\xb})+\Bb\Delta_{x}).
\end{align}
\subsubsection{d) Lyapunov Candidate for Translation Dynamics}
From the linearized control system developed at section 3, we use matrix $P$ to introduce the following Lyapunov candidate for translational dynamics
\begin{align}
\mathcal{V}_{1}=z_{1}^{T}Pz_{1}+2\int_{p_{eq}}^{e_{\xb}}{(\Bb K_{z}\satr_{\sigma}(\mu)-\Bb\Delta_{x})}\cdot d \mu.
\end{align}
The last integral term of the above equation is positive definite about the equilibrium point $e_{\xb}=p_{eq}$ where
\begin{align}
p_{eq}=[\frac{\Delta_{x}}{k_{z}},0,0,\cdots],
\end{align}
if 
\begin{align}
\delta< k_z\sigma,
\end{align}
considering the fact that $\sat_{\sigma}{y}=y$ if $y<\sigma$.
The time derivative of the Lyapunov function using the Leibniz integral rule is given by
\begin{align}\label{eqn:devrr}
\dot{\mathcal{V}_{1}}=\dot{z}_{1}^{T}Pz_{1}+z_{1}^{T}P\dot{z}_{1}+2\dot{e}_{\xb}\cdot(\Bb K_{z}\satr_{\sigma}(e_{\xb})-\Bb\Delta_{x}).
\end{align}
Since $\dot{e}_{\xb}^{T}=((P\mathds{B})^{T}z_{1})^{T}=z_{1}^{T}P\mathds{B}$ from \refeqn{exterm}, the above expression can be written as
\begin{align}\label{eqn:devvcv}
\dot{\mathcal{V}_{1}}=\dot{z}_{1}^{T}Pz_{1}+z_{1}^{T}P\dot{z}_{1}+2z_{1}^{T}P\mathds{B}(\Bb K_{z}\satr_{\sigma}(e_{\xb})-\Bb\Delta_{x}).
\end{align}
Substituting \refeqn{zdot1} into \refeqn{devvcv}, it reduces to
\begin{align}\label{eqn:beforsimp}
\dot{\mathcal{V}}_{1}=z_{1}^{T}(\mathds{A}^{T}P+P\mathds{A})z_{1}+2z_{1}^{T}P\mathds{B}(-\Bb X+\g(\xb,\dot{\xb})).
\end{align}
Let $c_{3}=2\|P\mathds{B}\Bb\|_{2}\in\Re$ and using $\mathds{A}^{T}P+P\mathds{A}=-Q$, we have
\begin{align}\label{eqn:test}
\dot{\mathcal{V}}_{1}\leq-z_{1}^{T}Qz_{1}+c_{3}\|z_{1}\|\|X\|+2z_{1}^{T}P\mathds{B}g(\xb,\dot{\xb}).
\end{align}
The second term on the right hand side of the above equation corresponds to the effects of the attitude tracking error on the translational dynamics. We find a bound of $X$, defined at \refeqn{Xdef}, to show stability of the coupled translational dynamics and rotational dynamics in the subsequent Lyapunov analysis. Since 
\begin{align}
f=\|A\|(e_{3}^{T}R_{c}^{T}Re_{3}), 
\end{align}
we have
\begin{align}\label{eqn:ssstr}
\|X\|\leq\|A\|\|(e_{3}^{T}R_{c}^{T}Re_{3})Re_{3}-R_{c}e_{3}\|.
\end{align}
The last term $\|(e_{3}^{T}R_{c}^{T}Re_{3})Re_{3}-R_{c}e_{3}\|$ represents the sine of the angle between $b_{3}=Re_{3}$ and $b_{3_{c}}=R_{c}e_{3}$, since $(b_{3_{c}}\cdot b_{3})b_{3}-b_{3_{c}}=b_{3}\times(b_{3}\times b_{3_{c}})$. The magnitude of the attitude error vector, $\|e_{R}\|$ represents the sine of the eigen-axis rotation angle between $R_{c}$ and $R$. Therefore, $\|(e_{3}^{T}R_{c}^{T}Re_{3})Re_{3}-R_{c}e_{3}\|\leq\|e_{R}\|$ in $D_{1}$. It follows that
\begin{align}
\begin{split}
\|(e_{3}^{T}R_{c}^{T}Re_{3})Re_{3}-R_{c}e_{3}\|&\leq\|e_{R}\|=\sqrt{\Psi(2-\Psi)}\\
&\leq\{\sqrt{\psi_1(2-\psi_1)}\triangleq\alpha\}<1,
\end{split}
\end{align}
therefore
\begin{align}
\begin{split}
\|X\|&\leq \|A\|\|e_{R}\|\\
&\leq\|A\|\alpha.
\end{split}
\end{align}
We also use the following properties
\begin{align}
\lambda_{\min}(Q)\|z_{1}\|^{2}\leq z_{1}^{T}Qz_{1}.
\end{align}
Note that $\lambda_{\min}(Q)$ is real and positive since $Q$ is symmetric and positive definite. Then, we can simplify \refeqn{test} as given
\begin{align}\label{eqn:ssddss}
\dot{\mathcal{V}}_{1} \leq -\lambda_{\min}(Q) \|z_{1}\| ^{2}+c_{3} \|z_{1}\| \|A\|\|e_{R}\|+2z_{1}^{T}P\mathds{B}\g(\xb,\dot{\xb}).
\end{align}
We find an upper boundary for
\begin{align}\label{eqn:AA}
A=-K_{x} \xb-K_{\dot{x}}\dot\xb -K_{z}\satr_{\sigma}(e_{\xb})+ M_{00}ge_3,
\end{align}
by defining
\begin{align}
\|M_{00}ge_{3}\|\leq B_{1},
\end{align} 
for a given positive constant $B_{1}$. We use the following properties where for any matrix $A\in\Re^{m\times n}$ 
\begin{align}
\|A\|_{2}\leq \sqrt{mn}\|A\|_{\max},
\end{align}
where $\|A\|_{\max}=\max\{a_{mn}\}$. The third term on the right hand side of \refeqn{AA} can be bounded as
\begin{align}
\|-K_{z}\satr_{\sigma}(e_{\xb})\|\leq \|K_{z}\|\|\satr_{\sigma}(e_{\xb})\|\leq\|K_{z}\| \sqrt{2n+3}\sigma,
\end{align}
where
\begin{align*}
\|K_{z}\|\leq \sqrt{3(2n+3)}\|K_{z}\|_{\max},
\end{align*}
and similarly 
\begin{align*}
&\|K_{\xb}\|\leq \sqrt{3(2n+3)}\|K_{\xb}\|_{\max},\\
&\|K_{\dot{\xb}}\|\leq \sqrt{3(2n+3)}\|K_{\dot{\xb}}\|_{\max}.
\end{align*}
We define $K_{max}, K_{z_{m}}\in\Re$
\begin{align*}
&K_{\max}=\max\{\|K_{\xb}\|_{\max},\|K_{\dot{\xb}}\|_{\max}\}\sqrt{3(2n+3)}, \\
&K_{z_{m}}=\sqrt{3}(2n+3)\|K_{z}\|_{\max},
\end{align*}
and then the upper bound of $A$ is given by
\begin{align}
\|A\| & \leq K_{\max}(\|\xb\|+\|\dot{\xb}\|)+\sigma K_{z_{m}}+B_{1}\\
&\leq 2K_{\max}\|z_{1}\|+(B_{1}+\sigma K_{z_{m}}),\label{eqn:normA}
\end{align}
and substitute \refeqn{normA} into \refeqn{ssddss}
\begin{align}\label{eqn:eq1}
\begin{split}
\dot{\mathcal{V}}_{1} \leq& -(\lambda_{\min}(Q)-2c_{3}K_{\max}\alpha) \|z_{1}\| ^{2}\\
&+c_{3}(B_{1}+\sigma K_{z_{m}})\|z_{1}\|\|e_{R}\|+2z_{1}^{T}P\mathds{B}\g(\xb,\dot{\xb}).
\end{split}
\end{align}
\subsubsection{e) Lyapunov Candidate for the Complete System}
Let $\mathcal{V}=\mathcal{V}_{1}+\mathcal{V}_{2}$ be the Lyapunov function for the complete system. The time derivative of $\mathcal{V}$ is given by
\begin{align}
\dot{\mathcal{V}}=\dot{\mathcal{V}}_{1}+\dot{\mathcal{V}}_{2}.
\end{align}
Substituting \refeqn{eq1} and \refeqn{eq2} into the above equation
\begin{align}
\begin{split}
\dot{\mathcal{V}}\leq& -(\lambda_{\min}(Q)-2c_{3}K_{\max}\alpha) \|z_{1}\| ^{2}+2z_{1}^{T}P\mathds{B}\g(\xb,\dot{\xb})\\
&+c_{3}(B_{1}+\sigma K_{z_{m}}) \|z_{1}\|\|e_{R}\|-\lambda_{m}(W_{2})\|z_{2}\|^{2},
\end{split}
\end{align}
and using $\|e_{R}\|\leq \|z_{2}\|$, it can be written as
\begin{align}\label{eqn:finalsimp}
\begin{split}
\dot{\mathcal{V}}\leq& -(\lambda_{\min}(Q)-2c_{3}K_{\max}\alpha) \|z_{1}\| ^{2}+2z_{1}^{T}P\mathds{B}\g(\xb,\dot{\xb})\\
&+c_{3}(B_{1}+\sigma K_{z_{m}})\|z_{1}\|\|z_{2}\|-\lambda_{m}(W_{2})\|z_{2}\|^{2}.
\end{split}
\end{align}
Also, the $2z_{1}^{T}P\mathds{B}\g(\xb,\dot{\xb})$ term in the above equation is indefinite. The function $\g(\xb,\dot{\xb})$ satisfies
\begin{align}
\frac{\|\g(\xb,\dot{\xb})\|}{\|z_{1}\|}\rightarrow 0\quad \mbox{as}\quad \|z_{1}\|\rightarrow 0
\end{align}
Then, for any $\gamma>0$ there exists $r>0$ such that
\begin{align}
\|\g(\xb,\dot{\xb})\|<\gamma\|z_{1}\|,\quad \forall\|z_{1}\|<r
\end{align}
so,
\begin{align}
2z_{1}^{T}P\mathds{B}\g(\xb,\dot{\xb})\leq 2\gamma\|P\|_{2}\|z_{1}\|^{2}.
\end{align}
Substituting the above equation into \refeqn{finalsimp}
\begin{align}
\begin{split}
\dot{\mathcal{V}}\leq &-(\lambda_{\min}(Q)-2c_{3}K_{\max}\alpha) \|z_{1}\| ^{2}+2\gamma\|P\|_{2}\|z_{1}\|^{2}\\
&+c_{3}(B_{1}+\sigma K_{z_{m}})\|z_{1}\|\|z_{2}\|-\lambda_{m}(W_{2})\|z_{2}\|^{2},
\end{split}
\end{align}
we obtain
\begin{align}
\dot{\mathcal{V}}\leq-z^{T}Wz+2\gamma\|P\|_{2}\|z_{1}\|^{2},
\end{align}
where $z=[z_{1},z_{2}]^{T}\in\Re^{2}$ and
\begin{align}
W=\begin{bmatrix}
\lambda_{\min}(Q)-2c_{3}K_{\max}\alpha&-\frac{c_{3}(B_{1}+\sigma K_{z_{m}})}{2}\\
-\frac{c_{3}(B_{1}+\sigma K_{z_{m}})}{2}&\lambda_{m}(W_{2})
\end{bmatrix}.
\end{align}
Also using $\|z_{1}\|\leq\|z\|$
\begin{align}
\dot{\mathcal{V}}\leq -(\lambda_{\min}(W)-2\gamma\|P\|_{2})\|z\|^{2}.
\end{align}
Choosing $\gamma<(\lambda_{\min}(W))/2\|P\|_{2}$, ensures that $\dot{\mathcal{V}}$ is negative semi-definite. This implies that the zero equilibrium of tracking errors is stable in the sense of Lyapunov and $\mathcal{V}$ is non-increasing. Therefore all of error variables $z_{1}$, $z_{2}$ and integral control terms $e_{I}$, $e_{\xb}$ are uniformly bounded. Also, from Lasalle-Yoshizawa theorem~\cite[Thm 3.4]{Kha96}, we have $z\rightarrow 0$ as $t\rightarrow \infty$.
\subsection{Proof for the high order terms derivations}
We approximate $\xi\in\Re^{3}$ by $e_{3}\times q$ and other high order terms. The following relations are considerable as it is illustrated in the Fig. \ref{appendixqxi}. 
\begin{figure}[h]
\centering
\includegraphics[width=2.1in]{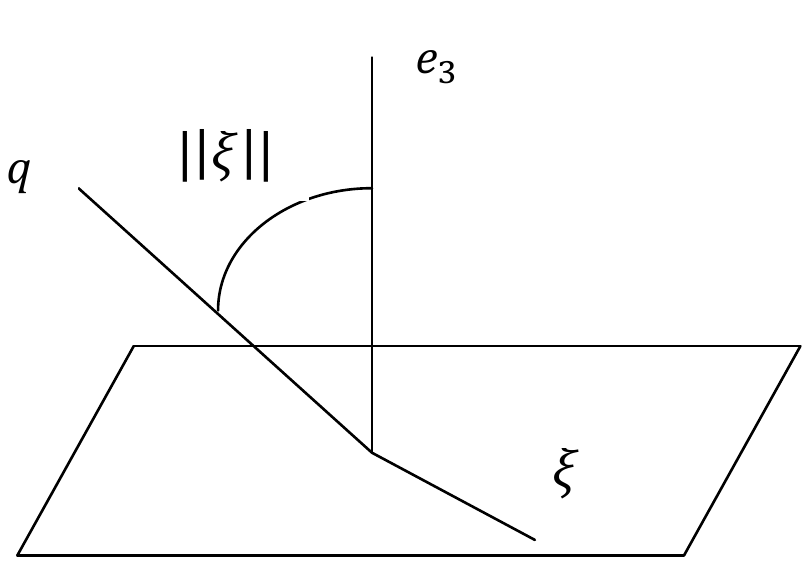}
\caption{$e_{3}$ and $q$, direction of each link}
\label{appendixqxi}
\end{figure}
\begin{align}
\sin\|\xi\|=\|q\times e_{3}\|,
\end{align}
and
\begin{align}
\frac{\xi}{\|\xi\|}=\frac{e_{3}\times q}{\|e_{3}\times q\|},
\end{align}
so
\begin{align}
\xi=\frac{e_{3}\times q}{\|e_{3}\times q\|}\sin^{-1}(\|e_{3}\times q\|).
\end{align}
Taking derivative and considering that derivative of $\|A\|$ is $\frac{A\cdot\dot{A}}{\|A\|}$
\begin{align}
\begin{split}
\dot{\xi}&=[\frac{e_{3}\times \dot{q}}{\|e_{3}\times q\|}-\frac{(e_{3}\times q)[(e_{3}\times q)\cdot(e_{3}\times \dot{q})]}{\|e_{3}\times q\|^3}]\times\\
&\sin^{-1}(\|e_{3}\times q\|)\\
&+\frac{e_{3}\times q}{\|e_{3}\times q\|}[\frac{(e_{3}\times q)\cdot(e_{3}\times\dot{q})}{\|e_{3}\times q\|}]\frac{1}{\sqrt{1-(\|e_{3}\times q\|)^2}},
\end{split}
\end{align}
after simplifying we would have
\begin{align}\label{eqn:derivative1}
\begin{split}
\dot{\xi}&=\frac{e_{3}\times\dot{q}}{\|e_{3}\times q\|}\sin^{-1}(\|e_{3}\times q\|)\\
&+(e_{3}\times q)[(e_{3}\times q)\cdot(e_{3}\times\dot{q})]k,
\end{split}
\end{align}
where $k\in\Re$ is scalar and defined as follow
\begin{align}
k=\frac{1}{\|e_{3}\times q\|^2\sqrt{(1-\|e_{3}\times q\|^2)}}-\frac{\sin^{-1}(\|e_{3}\times q\|)}{\|e_{3}\times q\|^3}.
\end{align}
The following relation is considerable 
\begin{align}
\dot{q}=\omega\times q,
\end{align}
where $\omega\in\Re^{3}$ is the angular velocity of each link. Using $a\times (b \times c)=b(a\cdot c)-c(a\cdot b)$
\begin{align}\label{eqn:q_dot}
e_{3}\times\dot{q}=e_{3}\times(q\times\omega)=\omega(e_{3}\cdot q)-q(e_{3}\cdot\omega),
\end{align}
and substituting Eq. \refeqn{q_dot} into Eq. \refeqn{derivative1}
\begin{align}\label{eqn:derivative22}
\begin{split}
\dot{\xi}&=\frac{\delta\omega(e_{3}\cdot q)}{\|e_{3}\times q\|}\sin^{-1}(\|e_{3}\times q\|)-\frac{q(e_{3}\cdot\delta\omega)}{\|e_{3}\times q\|}\sin^{-1}(\|e_{3}\times q\|)\\
&+[(e_{3}\times q)\cdot(\delta\omega(e_{3}\cdot q))](e_{3}\times q)k\\
&-[(e_{3}\times q)\cdot(q(e_{3}\cdot \delta\omega))](e_{3}\times q)k,
\end{split}
\end{align}
and using the fact that $a\cdot(a\times b)=0$ the last term on the right hand side vanishes, so
\begin{align}\label{eqn:derivative222}
\begin{split}
\dot{\xi}&=\frac{\delta\omega(e_{3}\cdot q)}{\|e_{3}\times q\|}\sin^{-1}(\|e_{3}\times q\|)\\
&+[(e_{3}\times q)\cdot(\delta\omega(e_{3}\cdot q))](e_{3}\times q)k\\
&-\frac{q(e_{3}\cdot\delta\omega)}{\|e_{3}\times q\|}\sin^{-1}(\|e_{3}\times q\|),
\end{split}
\end{align}
The first term on the right hand side of Eq. \refeqn{derivative222} can be simplified using Taylor series 
\begin{align}\label{eqn:tq}
q=\exp(\hat{\xi})e_{3}=(I+\hat{\xi}+\g(\xi))e_{3},
\end{align}
using the fact that $\hat{\xi}e_{3}\cdot e_{3}=0$ we can have
\begin{align}
q\cdot e_{3}=(e_{3}+\hat{\xi}e_{3}+\g(\xi))e_{3}=1+\g(\xi),
\end{align}
so considering Taylor series $\frac{\sin^{-1}(x)}{x}=1+\frac{1}{6}x^2+O(x^{n})$, the first term of the right hand side of Eq. \refeqn{derivative222} can be rewritten as
\begin{align}\label{eqn:tq}
q=\exp(\hat{\xi})e_{3}=(I+\hat{\xi}+\g(\xi))e_{3},
\end{align}
using the fact that $\hat{\xi}e_{3}\cdot e_{3}=0$ we can have
\begin{align}\label{eqn:e11}
q\cdot e_{3}=(e_{3}+\hat{\xi}e_{3}+\g(\xi))e_{3}=1+\g(\xi),
\end{align}
also
\begin{align}\label{eqn:e22}
q\times e_{3}=(e_{3}+\hat{\xi}e_{3}+\g(\xi))\times e_{3}=-\hat{e}_{3}^{2}\xi+\g(\xi),
\end{align}
and $\|A\|^{2}=A^{T}A$, so
\begin{align}\label{eqn:hote3q}
\begin{split}
\|e_{3}\times q\|^{2}&=(-\hat{e}_{3}^{2}\xi+\g(\xi))^{T}(-\hat{e}_{3}^{2}\xi+\g(\xi))\\
&=\|\xi\|^{2}+\g(\xi)=\g(\xi)
\end{split}
\end{align}
so 
\begin{align}\label{eqn:hqn}
\begin{split}
&\frac{(e_{3}\cdot q)}{\|e_{3}\times q\|}\sin^{-1}(\|e_{3}\times q\|)\\
&=(e_3\cdot q)(1+\frac{1}{6}(\|e_{3}\times q\|)^{2}+\cdots)\\
&=(1+\g(\xi))(1+\g(\xi))=1+\g(\xi).
\end{split}
\end{align}
The third term is simplified as follow using the fact that $\omega$ is normal to $q$ 
\begin{align}\label{eqn:fact}
\delta\omega\cdot q=0,
\end{align}
replacing $q=e_{3}+\hat{\xi}e_{3}+\g(\xi)$ into the above equation we would have
\begin{align}
(e_{3}+\hat{\xi}e_{3}+\g(\xi))\cdot \delta\omega=0,
\end{align}
where $g(\xi)$ is higher order terms
\begin{align}
e_{3}\cdot\delta\omega+\hat{\xi}e_{3}\cdot\delta\omega+\g(\xi)=0,
\end{align}
and
\begin{align}
e_{3}\cdot\delta\omega=-\hat{\xi}e_{3}\cdot\delta\omega+\g(\xi)=\g(\xi,\delta\omega)+\g(\xi)=\g(\xi,\delta\omega),
\end{align}
so, we can show that 
\begin{align}\label{eqn:set1}
\delta\omega\cdot e_{3}=\g(\xi,\delta\omega),
\end{align}
then, 
\begin{align}
\frac{q(e_{3}\cdot\delta\omega)}{\|e_{3}\times q\|}\sin^{-1}(\|e_{3}\times q\|)=\g(\xi,\delta\omega).
\end{align}
Finding the Taylor series of $k$
\begin{align}
\begin{split}
&\frac{1}{\|e_{3}\times q\|^2\sqrt{(1-\|e_{3}\times q\|^2)}}-\frac{\sin^{-1}(\|e_{3}\times q\|)}{\|e_{3}\times q\|^3}\\
&=\frac{1}{3}+\frac{3}{10}\|e_{3}\times q\|^2+\frac{15}{56}\|e_{3}\times q\|^4+\cdots,
\end{split}
\end{align}
and using \refeqn{hote3q}
\begin{align}
k=\frac{1}{3}+\frac{3}{10}\|e_{3}\times q\|^2+\frac{15}{56}\|e_{3}\times q\|^4+\cdots=\frac{1}{3}+\g(\xi).
\end{align} 
We use \refeqn{e22} to simplify the following term
\begin{align}\label{eqn:simple}
\begin{split}
(e_{3}\times q)\cdot\delta\omega&=(-\hat{e}_{3}^{2}\xi+g(\xi))\cdot\delta\omega\\
&=-\hat{e}_{3}^{2}\xi\cdot\delta\omega+\g(\xi)=\g(\xi,\delta\omega)+\g(\xi)\\
&=g(\xi,\delta\omega),
\end{split}
\end{align}
so, the second term on the right hand side of \refeqn{derivative222} becomes
\begin{align}
\begin{split}
&[(e_{3}\times q)\cdot(\delta\omega(e_{3}\cdot q))](e_{3}\times q)k\\
&=(e_3\cdot q)(e_{3}\times q)(\g(\xi,\delta\omega))(\frac{1}{3}+\g(\xi))=\g(\xi,\delta\omega).
\end{split}
\end{align}
Finally, we can approximate the $\dot{\xi}$ as given
\begin{align}
\begin{split}
\dot{\xi}&=\delta\omega(1+\g(\xi))+\g(\xi,\delta\omega)+\g(\xi,\delta\omega)\\
&=\delta\omega+\g(\xi,\delta\omega).
\end{split}
\end{align}
We can show that Eq. \refeqn{derivative222} can be expressed as presented 
\begin{align}
\dot{\xi}=\delta\omega(1+\g(\xi))+\g(\xi,\delta\omega)=\delta\omega+\g(\xi,\delta\omega).
\end{align}
Now, we start finding an expression for $\delta\dot{\omega}$ by taking derivative of the above expression
We take derivative of \refeqn{fact}
\begin{align}\label{eqn:set2}
\dot{q}\cdot\delta\omega+q\cdot\delta\dot{\omega}=0,
\end{align}
so
\begin{align}
\dot{q}\cdot(q\times\dot{q})+q\cdot\delta\dot{\omega}=0
\end{align}
and using $a\cdot{a\times b}=0$
\begin{align}
q\cdot\delta\dot{\omega}=0,
\end{align}
replacing $q=e_{3}+\hat{\xi}e_{3}+\g(\xi)$ into the above expression and simplifying
\begin{align}
e_{3}\cdot\delta\dot{\omega}=-\hat{\xi}e_{3}\cdot\delta\dot{\omega}+\g(\xi)=\g(\xi,\delta\omega),
\end{align}
so finally
\begin{align}\label{eqn:sss}
e_{3}\cdot\delta\dot{\omega}=\g(\xi,\delta\omega),
\end{align}
We take derivative of the $\dot{\xi}$ equation to find an expression for $\delta\dot{\omega}$. 
\begin{align}
\begin{split}
\ddot{\xi}&=\frac{\dot{q}(e_{3}\cdot\delta\omega)}{\|e_{3}\times q\|}\sin^{-1}(\|e_{3}\times q\|)\\
&+\frac{q(e_{3}\cdot\delta\dot{\omega})}{\|e_{3}\times q\|}\sin^{-1}(\|e_{3}\times q\|)+q(e_{3}\cdot\delta\omega)k\\
&+\frac{\delta\dot{\omega}(e_{3}\cdot q)}{\|e_{3}\times q\|}\sin^{-1}(\|e_{3}\times q\|)+\delta\omega(e_{3}\cdot\dot{q})\frac{\sin^{-1}(\|e_{3}\times q\|)}{\|e_{3}\times q\|}\\
&+\delta\omega(e_{3}\cdot q)[\frac{1}{\|e_{3}\times q\|\sqrt{1-(\|e_{3}\times q\|)^2}}\\
&-\frac{\sin^{-1}(\|e_{3}\times q\|)}{\|e_{3}\times q\|^2}]\frac{(e_{3}\times q)(e_{3}\times\dot{q})}{\|e_{3}\times q\|}\\
&+[(e_{3}\times\dot{q})\cdot\delta\omega](e_{3}\cdot q)(e_{3}\times q)k\\
&+[(e_{3}\times q)\cdot\delta\dot{\omega}](e_{3}\cdot q)(e_{3}\times q)k\\
&+[(e_{3}\times q)\cdot\delta\omega][(e_{3}\cdot \dot{q})(e_{3}\times q)k\\
&+(e_{3}\cdot q)(e_{3}\times \dot{q})k+(e_{3}\cdot q)(e_{3}\times q)\dot{k}],
\end{split}
\end{align}
The first line is higher order term based on the derivations at \refeqn{sss} and \refeqn{set1}. The last line is also higher order term based on the derivations at \refeqn{simple} and $\ddot{\xi}$ becomes
\begin{align}
\begin{split}
\ddot{\xi}&=\delta\dot{\omega}+\g(\xi,\delta\omega)\\
&+[(e_{3}\cdot q)(e_{3}\times q)k][\delta\omega(e_{3}\times \dot{q})\\
&+(e_{3}\times\dot{q})\cdot\delta\omega+(e_{3}\times q)\cdot\delta\dot{\omega}]\\
&+\delta\omega(e_{3}\cdot\dot{q})\frac{\sin^{-1}(\|e_{3}\times q\|)}{\|e_{3}\times q\|}.
\end{split}
\end{align}
We can show that the last line is higher order term as follow
\begin{align}
\begin{split}
e_{3}\cdot\dot{q}&=e_{3}\cdot(\delta\omega\times q)=-\delta\omega\cdot(e_{3}\times q)\\
&=\hat{e}_{3}^{2}\xi\cdot\delta\omega+\g(\xi)=\g(\xi,\delta\omega),
\end{split}
\end{align}
and
\begin{align}
\delta\omega\cdot(e_{3}\times\dot{q})=\delta\omega\cdot(\delta\omega(e_{3}\cdot q)-q(e_{3}\cdot\delta\omega))=\g(\xi,\delta\omega),
\end{align}
and
\begin{align}
\begin{split}
(e_{3}\times q)\cdot\delta\dot{\omega}&=(-\hat{e}_{3}^{2}\xi+\g(\xi))\cdot\delta\dot{\omega}\\
&=-\hat{e}_{3}^{2}\xi\cdot\delta\dot{\omega}+\g(\xi)=\g(\xi,\delta\omega),
\end{split}
\end{align}
so $\ddot{\xi}$
becomes
\begin{align}
\ddot{\xi}=\delta\dot{\omega}+\g(\xi,\delta\omega),
\end{align}
or
\begin{align}\label{eqn:sett}
\delta\dot{\omega}=\ddot{\xi}-\g(\xi,\delta\omega).
\end{align}

\bibliography{IJCAS}
\bibliographystyle{IEEEtran}

\biography{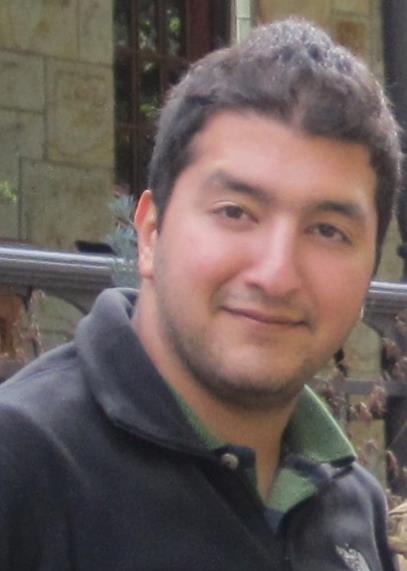}{Farhad A. Goodarzi}
	{received B.S. and M.S. degrees in Mechanical Engineering from Sharif University and Santa Clara University, CA in 2009 and 2011. Currently a Ph.D. candidate in ME department at The George Washington University. His research interests include control of complex systems and its application such as autonomous load transportation using multiple quadrotors.}
	\biography{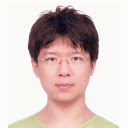}{Daewon Lee}
	{received the B.S., M.S. and Ph.D. degrees in Mechanical Engineering from Seoul National University. He is currently a Post doctoral fellow in mechanical and aerospace engineering department at The George Washington University. His research interests include control theory and its application to control of the quadrotor UAV's.}
	\biography{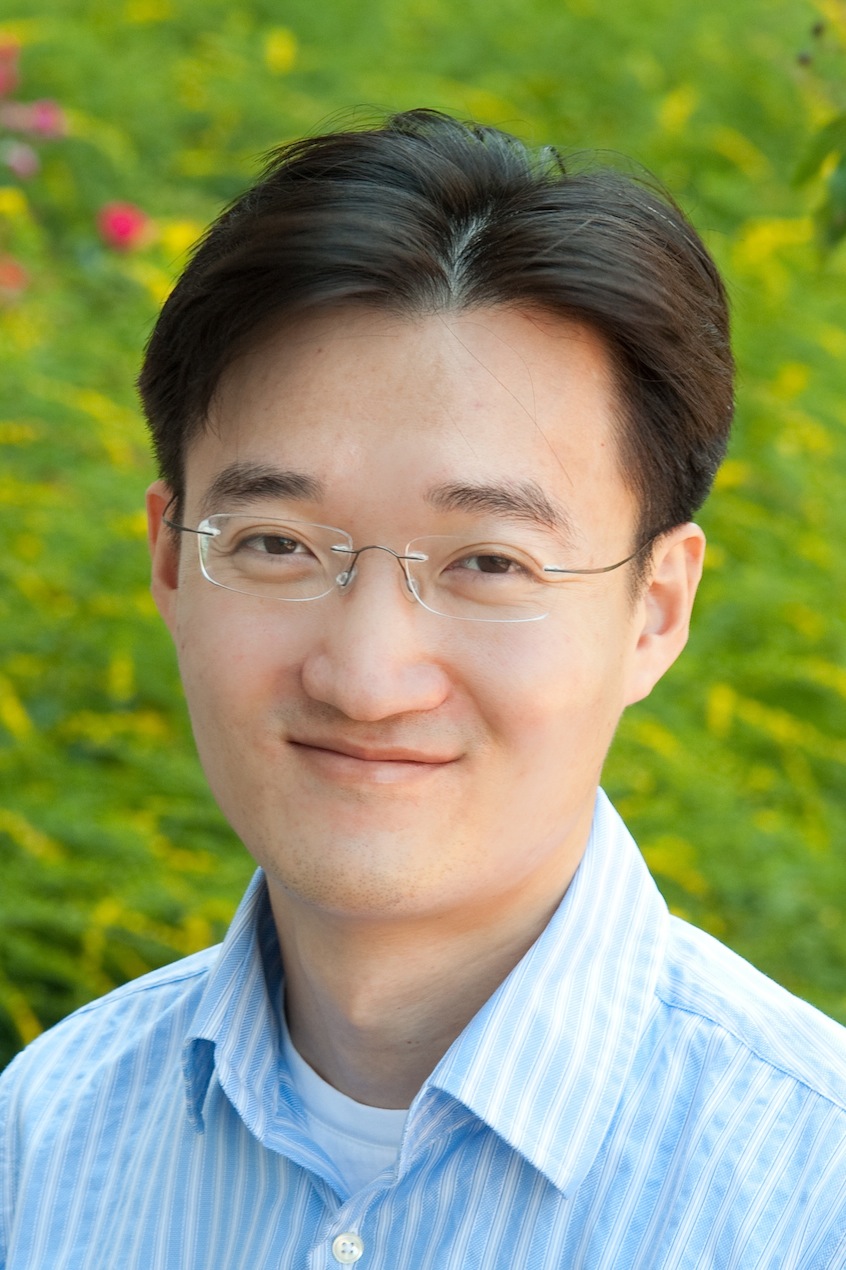}{Taeyoung Lee}
	{is an assistant professor of the Department of Mechanical and Aerospace Engineering at the George Washington University. He received his doctoral degree in Aerospace Engineering and his master's degree in Mathematics at the University of Michigan in 2008. His research interests include computational geometric mechanics and control of complex systems.}

\clearafterbiography\relax

\end{document}